\documentclass[12pt,a4paper,reqno]{amsart}
\usepackage[utf8]{inputenc}
\usepackage[T1]{fontenc}
\usepackage[foot]{amsaddr}
\usepackage{amsmath}
\usepackage{amsfonts}
\usepackage{amssymb}
\usepackage{algorithm}
\usepackage{algpseudocode}
\usepackage{xcolor}
\usepackage{soul}
\usepackage{graphicx}
\usepackage{wrapfig}
\usepackage{subcaption}
\usepackage[nopar]{lipsum}
\usepackage{amsthm}
\usepackage{amssymb}
\usepackage{euscript}
\usepackage{blkarray}
\usepackage{makecell, multirow, tabularx,booktabs}
\usepackage{enumitem}
\usepackage{mathtools}
\usepackage{soul}
\usepackage{amscd}
\usepackage{float}
\usepackage[pagebackref,colorlinks=true,
linkcolor=blue,
urlcolor=red,colorlinks,
citecolor=green]{hyperref}
\usepackage{cleveref}
\usepackage{geometry}
\usepackage{pgf,tikz,pgfplots}
\tikzstyle{none}=[]
\tikzstyle{new style 0}=[draw,circle,fill=white]
\tikzstyle{new edge style 1}=[draw,dashed]
\tikzstyle{new edge style 1}=[draw,dashed]
\usepackage{adjustbox}
\usepackage{setspace}
\usepackage[numbers,sort&compress]{natbib}

\allowdisplaybreaks[4]
\usetikzlibrary{calc}
\usetikzlibrary{patterns}
\pgfplotsset{compat=1.14}
\numberwithin{equation}{section}
\newtheorem{thm}{Theorem}[section]
\newtheorem{prop}[thm]{Proposition}
\newtheorem{lem}[thm]{Lemma}
\newtheorem{cor}[thm]{Corollary}
\newtheorem{exm}[thm]{Example}
\newtheorem{df}[thm]{Definition}
\newtheorem{rem}[thm]{Remark}

\makeatletter
\@namedef{subjclassname@1991}{Subject}
\@namedef{subjclassname@2000}{Subject}
\@namedef{subjclassname@2010}{Subject}
\@namedef{subjclassname@2020}{Subject}
\makeatother
\begin{document}
	\title{Spectral decomposition of hypergraph automorphism compatible matrices}
	\author[Banerjee]{Anirban Banerjee} 
	\email[Banerjee ]{\textit {{\scriptsize anirban.banerjee@iiserkol.ac.in}}}
	\author[ Parui]{Samiron Parui} 
	\email[ Parui ]{\textit {{\scriptsize  samironparui@gmail.com}}}
	\address[Banerjee]{Department of Mathematics and Statistics, Indian Institute of Science Education and Research Kolkata, Mohanpur-741246, India}
	
	\address[Parui]{School of Mathematical Sciences,  National Institute of Science Education and Research Bhubaneswar,  Bhubaneswar, Padanpur, Odisha 752050, India}

	
	\date{\today}
	

	\begin{abstract}
	This study explores the relationship between hypergraph automorphisms and the spectral properties of matrices associated with hypergraphs. 
For an automorphism $f$, an \( f \)-compatible matrices capture aspects of the symmetry, represented by \( f \), within the hypergraph. 
 First, we explore rotation, a specific kind of automorphism and find that the spectrum of any matrix compatible with a rotation can be decomposed into the spectra of smaller matrices associated with that rotation. 
We show that the spectrum of any \(f\)-compatible matrix can be decomposed into the spectra of smaller matrices associated with the component rotations comprising \( f \). 
Further, we study a hypergraph symmetry termed unit-automorphism, which induces bijections on the hyperedges, though not necessarily on the vertex set. We show that unit automorphisms also lead to the spectral decomposition of compatible matrices.
	\end{abstract}
 \maketitle
\section{Introduction}
A \emph{hypergraph} $H$ is an ordered pair $(V(H), E(H))$, where the vertex set, $V(H)$ is a non-empty set, and the hyperedge set $E(H)$ is a subset of $\mathcal{P}(V(H))\setminus \emptyset$,   where $\mathcal{P}(V(H))$ is the power set of $H$. Each element $v\in V(H)$ is called a vertex within $H$, and any $e\in E(H)$ is called a hyperedge within $H$. In this context, we focus on \emph{finite hypergraphs}, signifying hypergraphs with a finite vertex set. Given a vertex $v\in V(H)$, the star of the vertex $v$ is $E_v(H)=\{e\in E(H):v\in e\}$. 

A graph $G$ is a hypergraph such that $|e|=2$ for all $e\in E(H)$. 
The connection between the structural symmetry found in graphs and matrices associated with graphs is well studied. These symmetrical patterns are described through concepts such as graph homomorphism, graph covering, quotient graphs, twin vertices, graph automorphisms, and quotient matrices. Various studies, including  \cite{Steve-Butler-2010-edgecover,Trevisan_symm_simax_2016,Oscar-rojo-balanced-tree-2005,mehatari-banerjee-motif2015,Vladimir-R-Covering-automorphisms,Benjamin-webb-gen-equitable-2019}, delve into how these symmetrical features manifest in the spectra of matrices associated with graphs. Graph automorphisms are adjacency relation-preserving bijections on the vertex set, which serve as fundamental tools for representing and understanding symmetry within graphs. 
Interesting results regarding the relationship between the multiplicities of the graph-adjacency eigenvalue and the graph automorphism can be found in the renowned book of \cite[Section 3.8]{rowlishionn-simic-sgt}. \emph{Equitable partition} is another technique employed to study graph symmetry and adjacency spectra \cite{rowlishionn-simic-sgt,brouwer2011spectra,godsil-compactgraph-equit-part}. Using equitable partition, certain eigenvalues of the adjacency matrix can be readily computed for graphs with significant structural symmetry. Later, using graph automorphisms, eigenvalues of a larger class of matrices, called automorphism-compatible matrices, are proved similar to a direct sum of smaller matrices associated with graph automorphisms in \cite{equit_benjamin1,benjamin2,Benjamin-webb-gen-equitable-2019}. 

Recently, matrices have been used to study hypergraphs \cite{Banerjee-2021-hgmat,Sarkar-banerjee-2020,up2021,rodriguez2003laplacian,Swarup-panda-2022-hypergraph,cardoso-trevisan-2022-signless}. Our previous work shows some connections between hypergraph matrices and hypergraph symmetries \cite{unit,invariant_2024}. 
In this study, we explore how hypergraph automorphisms are related to the spectra of hypergraph matrices. 
We find that for any hypergraph automorphism $f$, the symmetry associated with $f$ can be traced by 
a specific class of hypergraph matrices referred to as $f$-compatible matrices.
We start with specific automorphisms named \emph{rotation} and show that the spectrum of any rotation-compatible matrix
can be decomposed into the spectra of smaller matrices associated with the rotation. 
We also show that any hypergraph automorphism can be represented as a product of rotations, and given any hypergraph automorphism $f$, the spectrum of any $f$-compatible matrix can be decomposed into the spectra of a family of smaller matrices related to the factor rotations of $f$.
We provide an upper bound for the number of simple eigenvalues of a rotation-compatible matrix. A \emph{unit} in a hypergraph is a maximal collection of vertices that have the same stars. Here, we introduce a graph symmetry, named \emph{unit-automorphism}, which may not be a bijection on the vertex set but always induces a bijection of the hyperedge set. We show that unit automorphisms also lead us to the spectral decomposition of a class of matrices compatible with unit automorphism.

\section{Hypergraph automorphism}
Hypergraph automorphisms are hyperedge-preserving bijections on the vertex set of a hypergraph.
\begin{df}[Hypergraph automorphism] A \emph{hypergraph automorphism} is a bijection $f:V(H)\to V(H)$ such that $e\in E(H)$ if and only if the image of $e$, $\hat{f}(e)=\{f(v):v\in e\}\in E(H)$. 
\end{df}
Therefore, any automorphism $f:V(H)\to V(H)$ of a hypergraph $H$ induces a bijection  $\hat{f}:E(H)\to E(H)$ defined by $\hat f(e)=\{f(v):v\in e\}$ for all $e\in E(H)$. Given a hypergraph $H$, we denote the collection of all the automorphisms of $H$ as $Aut(H)$. The collection $Aut(H)$ forms a group under the composition of automorphisms. Given any $f\in Aut(H)$, the matrix $P_{f}=\left(p_{uv}\right)$ is a square matrix whose row and columns are indexed by $V(H)$, and 
\[p_{uv}=\begin{cases}
    1&\text{~if~}u\overset{f}{\mapsto}v,\\
    0&\text{~otherwise}.
\end{cases}\]
The matrix $P_f$ is the matrix representation of $f$ since for each $x:V(H)\to\mathbb{C}$, we have $(P_fx)(u)=x(f(u))=(x\circ f)(u)$.
For any automorphism $f$ of $H$, a matrix $M_H =\left(m_{uv}\right)$, indexed by $V(H)$, is referred to as \emph{$f$-compatible} if 
$ m_{uv}=m_{f(u)f(v)}$ for all $u,v\in V(H)$. In the forthcoming example, we demonstrate the prevalence of automorphism-compatible matrices within the extant literature.
\begin{exm}\label{ex-hyp-mat}\rm Adjacency matrices associated with hypergraphs are compatible with hypergraph automorphism.

1. The hypergraph adjacency matrix described in \cite{bretto2013hypergraph} is a matrix $A^{(r)}_H=\left(d_{uv}\right)_{u,v\in V(H)}$ such that $d_{uv}=|E_u(H)\cap E_v(H)|$, the number of hyperedges containing both $u$, and $v$, for all $u,v(\ne u)\in V(H)$, and the diagonal entry $d_{uu}=0$ for all $u\in V(H)$. Given any automorphism $f$ of $V(H)$, since $d_{uv}=d_{f(u)f(v)}$, the matrix $A^{(r)}_H$ is an $f$-compatible matrix.

2. Another adjacency $A^{(b)}_H=\left(a_{uv}\right)_{u,v\in V(H)}$ is considered in \cite{Banerjee-2021-hgmat}. For all $u,v(\ne u)\in V(H)$, the $(u,v)$-th entry of $A^{(b)}$  is $a_{uv}=\sum\limits_{e\in E_u(H)\cap E_v(H)}\frac{1}{|e|-1}$, and the diagonal entry $a_{uu}=0$ for all $u\in V(H)$. Any automorphism $f$ of $H$ induces an edge cardinality-preserving bijection $\hat f_{uv}:E_{u}(H)\cap E_v(H)\to E_{f(u)}(H)\cap E_{f(v)}(H)$. Consequently, $ a_{uv}=a_{f(u)f(v)}$ for all $u,v(\ne u)\in V(H)$, and $A^{(b)}_H$ is an $f$-compatible matrix. 

3. The probability transition matrix $P_H=\left(p_{uv}\right)_{u,v\in V(H)}$ considered in \cite{Banerjee-2021-hgmat} is defined as $p_{uv}=\frac{1}{|E_u(H)|}\sum\limits_{e\in E_u(H)\cap E_v(H)}\frac{1}{|e|-1}$ for two distinct $u,v\in V(H)$, and the diagonal entry $p_{uu}=0$ for all $u\in V(H)$. Since $|E_u(H)|=|E_{f(u)}(H)|$ for any hypergraph automorphism $f$, the probability transition matrix $P_H$ is also $f$-compatible for any $f\in Aut(H)$.

4. Similarly, the hypergraph Laplacian matrices described in \cite{Banerjee-2021-hgmat, rodriguez2003laplacian,rodriguez2009laplacian} are  compatible with any hypergraph automorphism. The Laplacian $L^{(r)}_H=\left(l^{(r)}_{uv}\right)_{u,v\in V(H)}$ considered in \cite{rodriguez2003laplacian} is such that $l^{(r)}_{uv}=-|E_u(H)\cap E_v(H)|$ for all distinct $u,v\in V(H)$, and for $u=v$, the diagonal entry is $ \sum\limits_{v\ne(u)\in V(H)}|E_u(H)\cap E_v(H)|$. For any $f\in Aut(H)$, since $|E_u(H)\cap E_v(H)|=|E_{f(u)}(H)\cap E_{f(v)}(H)|$, the matrix $L^{(r)}_H$ is $f$-compatible. For a similar reason, the hypergraph Laplacian matrix $L^{(b)}_H=\left(l^{(b)}_{uv}\right)_{u,v\in V(H)}$ defined in \cite{Banerjee-2021-hgmat} is also $f$-compatible for all $f\in Aut(H)$. Here $l^{(b)}_{uv}=-\sum\limits_{e\in E_u(H)\cap E_v(H)}\frac{1}{|e|-1}$ for two vertices $u\ne v$, and for $u\in V(H)$, the diagonal entry is $l^{(b)}_{uu}=|E_u(H)|$.

5. The signless Laplacian matrix $Q=\left(q_{uv}\right)_{u,v\in V(H)}$ described in \cite{skpanda2023signless} is such that
$q_{uv}=\sum\limits_{e\in E_u(H)\cap E_v(H)}\frac{1}{|e|-1}$ if $u\ne v$, and for $u=v$, the diagonal entry is $q_{uu}=|E_u(H)|$. Given any $f\in Aut(H)$, since $E_{f(v)}(H)=\{\hat{f}(e):e\in E_v(H)\}$ we have $q_{uu}=q_{f(u)f(u)}$ for all $u\in V(H)$ and for two distinct vertices $u, v$, the collection $E_{f(u)}(H)\cap E_{f(v)}(H)=\{\hat{f}(e):e\in E_u(H)\cap E_v(H)\}$ with $|e|=|\hat{f}(e)|$. Thus, $q_{uv}=q_{f(u)f(v)}$, and $Q$ is $f$-compatible.

6. Let $H$ be a hypergraph and $\delta_{V(H)}:V(H)\to (0,\infty)$, and $\delta_{E(H)}:V(H)\to (0,\infty)$ be two positive valued functions. The matrix representations of the general adjacency $A_{(H,\delta_{V(H)},\delta_{E(H)})}=\left(\mathfrak{a}_{uv}\right)_{u,v\in V(H)}$, general Laplacian $L_{(H,\delta_{V(H)},\delta_{E(H)})}=\left(\mathfrak{l}_{uv}\right)_{u,v\in V(H)}$, and general signless Laplacian ${\mathcal Q}_{(H,\delta_{V(H)},\delta_{E(H)})}=\left(\mathfrak{q}_{uv}\right)_{u,v\in V(H)}$ considered in \cite{up2021} are defined as
$$ \mathfrak{a}_{uv}=\begin{cases}
    \frac{1}{\delta_{V(H)}(u)}\sum\limits_{e\in E_u(H)\cap E_v(H)}\frac{\delta_{E(H)}(e)}{|e|^2} &\text{~if~} u\ne v\\
    0&\text{~otherwise,~}
\end{cases} $$
$$\mathfrak{l}_{uv}=\begin{cases}
    -\frac{1}{\delta_{V(H)}(u)}\sum\limits_{e\in E_u(H)\cap E_v(H)}\frac{\delta_{E(H)}(e)}{|e|^2} &\text{~if~} u\ne v\\
    \frac{1}{\delta_{V(H)}(u)}\sum\limits_{e\in E_u(H)}\frac{\delta_{E(H)}(e)}{|e|}&\text{~otherwise,~}
\end{cases}$$
and $ \mathfrak{q}_{uv}=\frac{1}{\delta_{V(H)}(u)}\sum\limits_{e\in E_u(H)\cap E_v(H)}\frac{\delta_{E(H)}(e)}{|e|^2}$ for all $u,v\in V(H)$. Different variations of $\delta_{V(H)}$, and $\delta_{E(H)}$ lead to an uncountable number of Adjacency, Laplacian and signless Laplacian of the hypergraph $H$. Given any $f\in Aut(H)$, if $\delta_{V(H)}(v)=\delta_{V(H)}(f(v))$, and $\delta_{E(H)}(e)=\delta_{E(H)}(\hat{f}(e))$ for all $v\in V(H)$, and $e\in E(H)$ then, $A_{(H,\delta_{V(H)},\delta_{E(H)})}$, $L_{(H,\delta_{V(H)},\delta_{E(H)})}$, and ${\mathcal Q}_{(H,\delta_{V(H)},\delta_{E(H)})}$ are $f$-compatible.

\end{exm}
In the case of a graph, the matrix representation of a graph automorphism commutes with the graph adjacency matrix (see \cite[Proposition 3.8.1]{rowlishionn-simic-sgt}). In the subsequent result, we observe a similar fact about a hypergraph automorphism and any matrix that is compatible with the hypergraph automorphism.
\begin{prop}
    Let $H$ be a hypergraph and $f$ be an automorphism of $H$. Given any $f$-compatible matrix $M_H$, it follows that $M_HP_f=P_fM_H$.
\end{prop}
\begin{proof}
    For $u,v\in V(H)$, consider the $(u,v)$-th entry
    $(M_HP_f)_{uv}=m_{uf^{-1}(v)}$. and $(P_fM_H)_{uv}=m_{f(u)v}$. Since $M_H$ is $f$-compatible, $ (P_fM_H)_{uv}=m_{f(u)v}=m_{uf^{-1}(v)}=(M_HP_f)_{uv}$. Therefore, $M_HP_f=P_fM_H$.
\end{proof}
As a consequence, given $f\in Aut(H)$, if $\lambda$ is an eigenvalue of an $f$-compatible matrix $M_H$ with an eigenvector\footnote{Given a matrix $M$ with an eigenvalue $\lambda$, we denote the eigenspace of $\lambda$ as $V_{\lambda}(M)$.}  $x:V(H)\to\mathbb{C}\in V_\lambda(M_H)$ then $P_fx\in V_\lambda(M_H)$. Thus, if $\lambda$ is a simple eigenvalue of $M_H$ then $ V_\lambda(M_H)$ is a one-dimensional eigenspace and $P_fx=\alpha x$ for some $\alpha\in \mathbb{C}$. Thus,  $\alpha$ is an eigenvalue of $P_f$ and $x\in V_{\alpha}(P_f)$. Since $P_f$ is a permutation matrix, a natural number $m$ exists, such as $P_f^m=I$, the identity matrix. Therefore, the absolute value, $|\alpha|=1$.
\begin{prop}
    Let $H$ be a hypergraph on $n$ vertices and $f\in Aut(H)$. Given an $f$-compatible matrix $M_H$, if $\lambda$ is a simple eigenvalue of $M_H$ with eigenvector $x$, then there exists $\alpha\in \mathbb{C}$ with $|\alpha|=1$ such that $P_fx=\alpha x$ if $M_H$ is a real matrix with all its eigenvalues real, then $\alpha=\pm 1$.
\end{prop}
\begin{proof}
    Since $P_fM_H=M_HP_f$, and $M_Hx=\lambda x$, we have $M_H(P_fx)=\lambda (P_fx)$. Since $\lambda$ is a simple eigenvalue of $M_H$, we have $P_fx=\alpha x$ for some $\alpha\in\mathbb{C}$.  Since $P_f$ is a permutation matrix, $|\alpha|=1$. If $M_H$ is a real matrix with all its eigenvalues real, then all the eigenvectors can be assumed to have real entries. Thus, $x$ is a real vector, and $P_f$ is a real matrix, thus
    $\alpha$ is a real number, and thus, $\alpha=\pm 1$.
\end{proof}
Therefore, for some $f\in Aut(H)$, if the $f$-compatible matrix $M_H$  is a real matrix with all its eigenvalues real and the eigenvectors of  $M_H$ form an eigenbasis of  $\mathbb{R}^n$, then for any simple eigenvalue $\lambda$ with eigenvector $x$ we have $P_f^2x=x$.  Since the eigenvectors of  $M_H$ form an eigenbasis of  $\mathbb{R}^n$, if all the eigenvalues of  $M_H$ are simple, then $P_f^2=I$ and the order of the automorphism $f$ is at most $2$. Thus, we have the following result, similar to \cite[Theorem 3.8.4]{rowlishionn-simic-sgt}.  
\begin{prop}
    Let $H$ be a hypergraph on $n$ vertices and $f\in Aut(H)$. If there exists an $f$-compatible matrix  $M_H$ with real entries such that
    all its eigenvalues of $M_H$ are real, and
simple eigenvalue then 
  the order of $f$ is at most $2$.
\end{prop}
\begin{proof}
    Suppose $H$ is a hypergraph on $n$ vertices and $\lambda_1,\ldots,\lambda_n$ are all the eigenvalues of $M_H$ and the corresponding eigenbasis of $\mathbb{R}^n$ is $\{x_1,\ldots,x_n\}$. Since $\lambda_i$ is a simple eigenvalue with eigenvector $x_i$, we have $P_f^2x_i=x_i$ for all $i=1,2,\ldots,n$, and consequently $P_f^2=I$ for all $f\in Aut(H)$. This completes the proof. 
\end{proof}
Therefore, given a hypergraph $H$ and $f\in Aut(H)$ with the order of $f$ is more than $2$. If $M_H$ is a matrix with real entries with all its eigenvalues real and the eigenvectors of  $M_H$ form an eigenbasis of  $\mathbb{R}^n$, and $M_H$ is compatible with $f$, then $M_H$ has at least one eigenvalue of multiplicity more than $1$.

The multiplicities of the eigenvalues of graph adjacency encode information about the automorphism group of the graph (c.f. \cite[Proposition 3.8.5]{rowlishionn-simic-sgt}). Now, we explore the similar relation between the automorphisms of hypergraphs and automorphism-compatible matrices associated with hypergraphs. 
Suppose that $M_H$ is an $f$-compatible matrix. Thus, $P_fM_H=M_HP_f$. If $M_H$ is diagonalizable, that is, there is a non-singular matrix $U$ such that $U^{-1}M_HU=D$, where $D$ is the diagonalization of $M_H$, then
$U^{-1}P_fUD=DU^{-1}P_fU$.  Now if  $n_1,\ldots,n_m$ are multiplicities of the eigenvalues of $M_H$, then 
$U^{-1}P_fU$ is a block diagonal matrix with the orders of the diagonal blocks are $n_1,\ldots,n_m$. Thus, we have the following result. 
\begin{prop}
    Let $H$ be a hypergraph, and $f\in Aut(H)$. If  $M_H$ is an $f$-compatible diagonalizable matrix and the multiplicities of the distinct eigenvalues of $M_H$ are $n_1,\ldots,n_m$, then $P_f$ is similar to a block-diagonal matrix with block sizes  $n_1,\ldots,n_m$. 
\end{prop}
\begin{proof}
   Let $U$ be  an invertible matrix, such that   $U^{-1}M_HU=D$, the diagonalization of $M_H$.  Since $M_H$ is $f$-compatible, $U^{-1}P_fUD=DU^{-1}P_fU$.
    Suppose that $Q=(q_{ij})=U^{-1}P_fU$
    and $D=diag(d_1,\dots,d_n)$, then  we have 
    $d_iq_{ij}=q_{ij}d_j$. Thus, if $\lambda_i$ is an eigenvalue of $M_H$ of multiplicity $n_i$, then $\lambda_i$ corresponds to a diagonal block of $Q$ of size $n_i$. So  
    $Q$ is a block diagonal matrix with diagonal blocks  $B_1,\ldots,B_m$ of size $n_1,\ldots,n_m$, respectively. This completes the proof.
\end{proof}
Given a hypergraph $H$, if  a matrix $M_H$ is compatible with all $f\in Aut(H)$ then $Aut(H)$ is isomorphic to a subgroup of $\mathcal{O}(n_1)\times\ldots\times\mathcal{O}(n_m)$ where $\mathcal{O}(k)$ is the multiplicative group of $k\times k$ orthogonal matrices, where $n_1,\ldots,n_m$ are the multiplicities of the distinct eigenvalues of $M_H$.
\subsection{Equitable partition}
The results we have presented so far show relations between hypergraph automorphisms and multiplicities of eigenvalues of automorphism-compatible matrices. Now, we use \emph{equitable partition} to explore further the relation between hypergraph symmetries and the eigenvalues of the matrices associated with hypergraphs.
Any automorphism $f$ of a hypergraph $H$ induces an equivalence relation $\mathfrak{R}_f$ on $V(H)$ such that $\mathfrak{R}_f=\{(u,v)\in V(H)\times V(H):v=f^i(u)\text{~for some integer~}$i$\}$. For any $v\in V(H)$, the $\mathfrak{R}_f$-equivalence class $O_f(v)$ is called the $f$-orbit of $v$. That is $O_f(v)=\{f^i(v):\text{for all integer }i\}$. Since equivalence classes form a partition, we have the partition of the vertex set $V(H)=\bigcup\limits_{v\in V(H)}O_f(v)$. We refer to this partition as the \emph{$f$-orbit partition of} $V(H)$.

For any square matrix \(M=(m_{uv})_{u,v\in V}\) with rows and columns indexed by the finite set \(V=\{1,2,3,\ldots, n\}\), an \emph{equitable partition} of the matrix \(M\) is defined as a disjoint partition \(V=V_1\cup V_2\cup\ldots\cup V_m\) of \(V\). This partition satisfies the condition that, for any \(i,j\in\{1,2,\ldots,m\}\),
\[\sum\limits_{w\in V_j}m_{uw}=\sum\limits_{w\in V_j}m_{vw}\]
for any pair of arbitrary indices \(u,v\in V_i\). 
In the case of graph automorphism, the orbits form an equitable partition of the vertex set (\cite[Proposition 3.2]{equit_benjamin1}). We show a similar fact for hypergraph automorphism in the following result.
 \begin{prop}\label{aut-equit}
		Let $H$ be a hypergraph. Suppose that $f$ is an automorphism in $H$. The $f$-orbit partition is an equitable partition for any $f$-compatible matrix $M_H$.
\end{prop}
\begin{proof}
	Let $O_f(u)$, and $O_f(v)$ be two $f$-orbits. For $u_1,u_2\in O_f(u)$, without loss of generality, we can assume that there exists a non-negative integer $i$ such that $u_2=f^i(u_1)$. Since $M_H$ is $f$-compatible, and the restriction of $ f^i$ on $O_f(v)$ is a bijection on $O_f(v)$,
		$$ \sum\limits_{w\in O_f(v)}m_{u_1w}=\sum\limits_{w\in O_f(v)}m_{f^i(u_1)f^i(w)}=\sum\limits_{w^\prime\in O_f(v)}m_{u_2w^\prime}.$$
		This completes the proof.
	\end{proof}
 Thus if $\mathcal{O}_f$ is the collection of all the orbits then for all $i,j\in \mathcal{O}_f$, $\sum\limits_{w\in j}m_{uw}=b_{ij}$ is a constant for all $u\in V_i $. Therefore, we define the \emph{$f$-orbit quotient matrix} $[M_H/\mathcal{O}_f]=\left(b_{ij}\right)_{i,j\in \mathcal{O}_f}$ of $M_H$. Before going into the next result, for any $x:\mathcal{O}_f\to\mathbb{C}$, we define $x_{\mathcal{O}_f:V(H)}\to \mathbb C$ as
		$x_{\mathcal{O}_f}(v)=x(O_f(v))$.
      \begin{prop}
 \label{lem-orbit}
		Let $H$ be a hypergraph. Suppose that $f$
		is an automorphism of $H$. For any $f$-compatible matrix $M_H$, each eigenvalue of $[M_H/\mathcal{O}_f] $ is an eigenvalue of $M_H$.
	\end{prop}
 \begin{proof}
		Let $[M_H/\mathcal{O}_f]x=\lambda x$. 
   If $O_f(v)=i\in \mathcal{O}_f$ for some $v\in V(H)$ then
		\begin{align*}
			(M_Hx_{\mathcal{O}_f})(v)=\sum\limits_{w\in V(H)}m_{vw}x_{\mathcal{O}_f}(w)=\sum\limits_{j\in \mathcal{O}_f}b_{ij}x(j)=([M_H/\mathcal{O}_f]x)(i).
		\end{align*}
		Therefore, for all $v\in V(H)$,
		$$(M_Hx_{\mathcal{O}_f})(v)=([M_H/\mathcal{O}_f]x)(O_f(v))=\lambda (x)(O_f(v))=\lambda (x_{\mathcal{O}_f})(v).$$
		This completes the proof.
	\end{proof}
\subsection{Rotation}
Now, we delve deeper by employing a specific class of hypergraph automorphisms to study eigenvalues and eigenvectors of automorphism-compatible matrices.

\begin{df}[Rotation of order $n$]Let $H$ be a hypergraph and $n(\ge 2)\in\mathbb{N}$. A hypergraph automorphism $f$ is called a rotation of order $n$ in $H$, if $V(H)$ can be partitioned into disjoint sets such that $V(H)=U_0\cup U_1\cup\ldots U_{n-1}\cup X$ such that $f(v)=v$ for all $v\in X$ and $U_{i+1}=\{f(v):v\in U_i\}$ for all $i=0,\ldots,n-2$ and $U_{0}=\{f(v):v\in U_{n-1}\}$. We refer 
		$U_i$ is the \textit{$i$-th component of the rotation $f$} for $i=0,\ldots,n-1$, and $X$ as the \textit{invariant set under the rotation $f$}.
	\end{df}
Given any rotation $f$ of order $n$ with invariant set $X\subseteq V(H)$,  $\Gamma_f=\{f^i:i=0,1,2,\ldots,n-1\}$ is a collection of rotations and for each $i$, the invariant set of $f^i$ is $X$. For all $i,j=1,2,\ldots,n-1$, we have $f^i\circ f^j=f^{[i+j]_n}\in \Gamma_f$, where $[i+j]_n$ is the remainder when $i+j$ is divided by $n$. The inverse ${(f^i)}^{-1}=f^{n-i}\in \Gamma_f$ for all $i=1,2,\ldots,n-1$, and $f^0$ is the identity automorphism of $H$. Therefore, $\Gamma_f$ is a subgroup of the $Aut(H)$.

\begin{figure}
    \centering
   \begin{tikzpicture}[scale=0.5]
	
		\node [style=none] (0) at (0, 0) {};
		\node [style=none] (1) at (2.25, 0) {};
		\node [style=none,scale=0.6] (2) at (3.25, 0) {9};
		\node [style=none] (5) at (4.5, 0) {};
		\node [style=none] (6) at (-1.25, 0) {};
		\node [style=none] (7) at (-3.5, -3.25) {};
		\node [style=none] (8) at (0.5, 1.25) {};
		\node [style=none] (9) at (0.75, -0.75) {};
		\node [style=none] (10) at (-3.5, 3.5) {};
		\node [style=none] (11) at (-2, -1.5) {};
		\node [style=none] (12) at (-2.5, -2.25) {};
		\node [style=none] (13) at (-1.75, 2.25) {};
		\node [style=none] (14) at (-2.75, 2.25) {};
		\node [style=none] (15) at (-3.25, 3) {};
		\node [style=none] (16) at (4, -0.25) {};
		\node [style=none] (17) at (-1.25, 1.5) {};
		\node [style=none] (18) at (1.5, 0.25) {};
		\node [style=none] (19) at (-3.5, -2.25) {};
		\node [style=none] (20) at (3.75, 1) {};
		\node [style=none] (21) at (-1.5, -0.75) {};
		\node [style=none] (22) at (1.25, 0.5) {};
		\node [style=none] (23) at (-2.5, 3.5) {};
		\node [style=none] (24) at (-0.75, 2.75) {};
		\node [style=none] (25) at (-0.75, -2) {};
		\node [style=none] (26) at (-2.5, -3) {};
		\node [style=none] (27) at (1.75, 3.5) {};
		\node [style=none] (28) at (1.5, -3) {};
		\node [style=none] (29) at (-3.5, 0.25) {};
		\node [style=none,scale=0.6] (30) at (0, 0) {1};
		\node [style=none,scale=0.6] (31) at (-1.75, 2.25) {2};
		\node [style=none,scale=0.6] (32) at (-2.75, 2.25) {3};
		\node [style=none,scale=0.6] (33) at (1.75, 3.5) {10};
		\node [style=none,scale=0.6] (34) at (-3.5, 0.25) {4};
		\node [style=none,scale=0.6] (35) at (-2, -1.5) {5};
		\node [style=none,scale=0.6] (36) at (-2.5, -2.25) {6};
		\node [style=none,scale=0.6] (37) at (1.5, -3) {7};
		\node [style=none,scale=0.6] (38) at (2.25, 0) {8};
		\node [style=none] (39) at (0.5, 6.25) {};
		\node [style=none] (40) at (-7.25, -2) {};
		\node [style=none] (41) at (5.5, -4.5) {};
		\node [style=none] (42) at (0.5, 0.75) {};
		\node [style=none] (43) at (-0.75, 0.5) {};
		\node [style=none] (44) at (-0.25, -0.5) {};
		\node [style=none,scale=0.75] (45) at (5, 0) {e};
		\node [style=none,scale=0.75] (46) at (-4, 3) {f};
		\node [style=none,scale=0.75] (47) at (-4, -3) {g};
		\node [style=none,scale=0.75] (48) at (1.25, 5.25) {i};
		\node [style=none,scale=0.75] (49) at (-4.5, -0.25) {j};
		\node [style=none,scale=0.75] (50) at (1.75, -4.25) {h};
		\node [style=none,scale=0.85,scale=0.9] (51) at (-5, 4) {$U_0$};
		\node [style=none,scale=0.85,scale=0.9] (52) at (-2, -5) {$U_1$};
		\node [style=none,scale=0.85,scale=0.9] (53) at (5.25, 2.75) {$U_2$};
		\node [style=none,scale=0.85] (54) at (4.5, -5.5) {H};
		\node [style=none,,scale=0.75] (55) at (19, 0) { \centering$A^{(r)}_H=$
     \begin{tabular}{|c||c|ccc|ccc|ccc|}\hline
        & X &  &$U_0$   &  &  &$U_1$  &  &  &$U_2$ & \\\hline\hline
      X&0 & 1 & 1 & 0 & 1 & 1 & 0 & 1 & 1 & 0\\\hline
      &1 & 0 & 3 & 1 & 1 & 1 & 0 & 1 & 1 & 1\\
     $ U_0$&1 & 3 & 0 & 1 & 1 & 1 & 0 & 1 & 1 & 1\\ 
      &0 & 1 & 1 & 0 & 1 & 1 & 0 & 0 & 0 & 0\\\hline
      &1 & 1 & 1 & 1 & 0 & 3 & 1 & 1 & 1 & 0\\ 
     $ U_1$&1 & 1 & 1 & 1 & 3 & 0 & 1 & 1 & 1 & 0\\
      &0 & 0 & 0 & 0 & 1 & 1 & 0 & 1 & 1 & 0\\\hline
      &1 & 1 & 1 & 0 & 1 & 1 & 1 & 0 & 3 & 1\\
     $ U_2$&1 & 1 & 1 & 0 & 1 & 1 & 1 & 3 & 0 & 1\\
      &0 & 1 & 1 & 0 & 0 & 0 & 0 & 1 & 1 & 0 \\ \hline
     \end{tabular}};
	
		\draw (5.center)
			 to [bend left=90, looseness=0.75] (6.center)
			 to [bend left=90, looseness=0.75] cycle;
		\draw (7.center)
			 to [bend left=90, looseness=0.75] (8.center)
			 to [bend left=90, looseness=0.75] cycle;
		\draw [bend left=90, looseness=0.75] (9.center) to (10.center);
		\draw [bend left=90, looseness=0.75] (10.center) to (9.center);
		\draw (15.center)
			 to [in=225, out=-105, looseness=1.25] (17.center)
			 to [bend left=90, looseness=2.50] (18.center)
			 to [bend right=90] (16.center)
			 to [bend right=90, looseness=1.50] cycle;
		\draw (19.center)
			 to [bend left=75, looseness=1.25] (21.center)
			 to [bend right=90, looseness=2.50] (22.center)
			 to [bend left=90] (20.center)
			 to [bend left=90, looseness=1.50] cycle;
		\draw (25.center)
			 to [bend left=45, looseness=1.25] (24.center)
			 to [bend right=60] (23.center)
			 to [bend right=45, looseness=1.25] (26.center)
			 to [bend right=75, looseness=1.25] cycle;
		\draw [style=new edge style 1,pattern=dots,pattern color=gray!40!white] (40.center)
			 to [bend left, looseness=0.75] (39.center)
			 to (43.center)
			 to cycle;
		\draw [style=new edge style 1,pattern=dots,pattern color=gray!40!white] (41.center)
			 to (44.center)
			 to (40.center)
			 to [bend right=15, looseness=1.25] cycle;
		\draw [style=new edge style 1,pattern=dots,pattern color=gray!40!white] (41.center)
			 to [bend right, looseness=1.25] (39.center)
			 to (42.center)
			 to cycle;
	\draw[dashed] (0,0) circle [radius=20pt];
 \node[scale=0.75]at (0.35,0.25){X};
\end{tikzpicture}

    \caption{A rotation in a hypergraph $H$ and adjacency matrix $A^{(r)}_H$ is compatible with this rotation.}
    \label{fig:rot}
\end{figure}
\begin{exm}\label{ex-hyp-rot}\rm
    Consider the hypergraph $H$ illustrated in the \Cref{fig:rot}. The vertex set, $V(H)=\{1,2,\ldots,10\}$, and the hyperedge set, $E(H)=\{e,f,g,h,i,j\}$ where $e=\{1,8,9\}$, $f=\{1,2,3\}$, $g=\{1,5,6\}$, $h=\{5,6,7,8,9\}$, $i=\{8,9,10,3,2\}$, and $j=\{2,3,4,5,6\}$ are the hyperedges of $H$. Given the function $f:V(H)\to V(H)$ defined as\footnote{ For any function on a finite domain $\{1,\ldots,n\}$, we represent it as an array with two-row, $f=\left(\begin{smallmatrix}
       i:& 1&\ldots&n\\
       f(i):& f(1)&\ldots&f(n)
    \end{smallmatrix}\right)$, in which the first and second entries of a column are $i$, and $f(i)$ respectively, for all $i=1,2,\ldots,n$. } 
    $$f=\left(\begin{smallmatrix}
       i:& 1&2&3&4&5&6&7&8&9&10\\
       f(i):& 1&5&6&7&8&9&10&2&3&4
    \end{smallmatrix}\right),$$
    the vertex set can be partitioned as $V(H)=U_0\cup U_1\cup U_2\cup X$ such that $f(U_0)=U_1$, $f(U_1)=U_2$,  $f(U_2)=U_0$, and $f(X)=X$. 
    This bijection $f$ is a rotation in $H$ of order $3$ with the invariant set $X$.
    \end{exm}
 \subsection{Rotation and eigenvalues of rotation-compatible matrices}
  Let $a_0,a_1,\ldots,a_{m-1}\in \mathbb{C}$. A circulant matrix with $[\begin{smallmatrix}
  a_0&a_1&\ldots&a_{n-1}   
 \end{smallmatrix}]$ as its first row is an $n\times n$ matrix 
 $C=\left[\begin{smallmatrix}
  a_0&a_1&\ldots&a_{n-1}\\
  a_{n-1}&a_0&\ldots&a_{n-2}\\
  \vdots&\vdots&\ddots&\vdots\\
  a_1&a_2&\ldots&a_0
 \end{smallmatrix}\right]$. If we replace each $a_i$ by an $m\times m$ matrix $A_i$, then the obtained $nm\times nm$ matrix $B=\left[\begin{smallmatrix}
  A_0&A_1&\ldots&A_{n-1}\\
  A_{n-1}&A_0&\ldots&A_{n-2}\\
  \vdots&\vdots&\ddots&\vdots\\
  A_1&A_2&\ldots&A_0
 \end{smallmatrix}\right]$ is called a block-circulant matrix of order $nm$. We see in \Cref{fig:rot}, the matrix $A_H^{(r)}$ has a block-circulant submatrix corresponding to the vertices in $ U_0\cup U_1\cup U_2$.  
 Any matrix compatible with a rotation of order $n$ contains a block-circulant submatrix of order $mn$, where $m$ is the cardinality of the $0$-th component of the rotation. 
 If we consider the block circulant matrix $B$, then for any $n$-th root of unity $\omega$, 
each eigenvalue of the $m\times m$ matrix $B_\omega=\sum\limits_{i=0}^{n-1}\omega^iA_i$ is an eigenvalue of $B$ (c.f. \cite[Section-4]{kaveh2011block}). Now, we introduce the following notion.
 \begin{df}[Rotation matrix]
		Let $H $ be a hypergraph. Suppose that $f$ is a rotation of order $n$ in $H$ with the $0$-th component of $f$ is $U_0$, and $M_H=\left(m_{uv}\right)_{u,v\in V(H)}$ is a $f$-compatible matrix associated with $H$. Let $\omega$ be an $n$-th root of unity. 
  The \emph{$w$-rotation matrix} of $M_H$ associated with $f$ is a matrix $M^{\omega}_H(f,U_0)=\left(r_{uv}\right)_{u,v\in U_0}$, indexed by $U_0$, is defined as 
		$$r_{uv}=\sum\limits_{i=0}^{n-1}\omega^im_{uf^i(v)}.
		$$ 
	\end{df}
 \begin{exm}\label{rot-mat}\rm
     In the \Cref{ex-hyp-mat}, we see the matrix $A^{(r)}_H$ is compatible with hypergraph automorphism. The \Cref{fig:rot} illustrates the hypergraph $H$ considered in \cref{ex-hyp-rot}, and the matrix $A^{(r)}_H$. The rotation $f$ of order $3$ given in \Cref{ex-hyp-rot} is such that $A_H^{(r)}$ is compatible with $f$. Thus, for any root $\omega$ of the equation $x^3=1$, the rotation matrix ${A^{(r)}}^\omega_H(f,U_0)$ of $A^{(r)}_H$ is
     $${A^{(r)}}^\omega_H(f,U_0)=\left[\begin{smallmatrix}
         0 & 3 & 1\\
         3 & 0 & 1\\
         1 & 1 & 0
     \end{smallmatrix}\right]+\omega\left[\begin{smallmatrix}
         1 & 1 & 0\\
         1 & 1 & 0\\
         1 & 1 & 0
     \end{smallmatrix}\right]+\omega^2\left[\begin{smallmatrix}
         1 & 1 & 1\\
         1 & 1 & 1\\
         0 & 0 & 0
     \end{smallmatrix}\right].$$  
 \end{exm}
In the existing literature, it is well-documented that the eigenvalues of a block-circulant matrix can be derived through the computation involving smaller matrices (see, for example, \cite[Section-4]{kaveh2011block}, \cite[Lemma 2.4]{equit_benjamin1}). The subsequent theorem presented herein uses analogous methodologies to demonstrate that, in the case of a rotation-compatible matrix $M_H$, its eigenvalues can be effectively determined using the rotation matrices associated with $M_H$.
\begin{thm}[Spectral inheritance due to rotation]\label{inheritrotation}Let $H$ be a hypergraph. Suppose that $f$ is a rotation of order $n$ with $0$-th component $U_0$, and $M_H$ is a $f$-compatible matrix associated with $H$. Each eigenvalue of $M^{\omega}_H(f,U_0) $ is an eigenvalue of $ M_H$ for any $n$-th root of unity $\omega\ne1$.
	\end{thm}
 \begin{proof}
		Since $f$ is a rotation with the $0$-th component $U_0$, $V(H)$ can be partitioned into disjoint sets such that 
		$$V(H)=U_0\cup U_1\cup\ldots U_{n-1}\cup X,$$ and $f(v)=v$ for some $v\in X$ and $U_{i+1}=\{f(v):v\in U_i\}$ for all $i=0,\ldots,n-2$ and $U_{0}=\{f(v):v\in U_{n-1}\}$.
		
		Suppose that $\lambda $ is an eigenvalue of $M_H(f,U_0) $ with eigenvector $$x:U_0\to\mathbb{C}.$$
		We define $x_{\omega}:V(H)\to\mathbb{C}$ as
		$$x_\omega(v)=\begin{cases}
			\omega^i x(f^{-i}(v))&\text{~if~}v\in U_i,\\
			0&\text{~if~} v\in X.
		\end{cases}$$

  Thus, for all $u\in U_j$, there exists $u^\prime\in U_0$ such that $u=f^j(u^\prime)$. Therefore,
		\begin{align*}
			(M_Hx_{\omega})(u)&=\sum\limits_{v\in V(H)}m_{uv}x_{\omega}(v)=\sum\limits_{i=0}^{n-1}\sum\limits_{v\in U_0}m_{uf^i(v)}x_{\omega}(f^i(v))+\sum\limits_{v\in X}m_{uv}0\\
			&=\sum\limits_{i=0}^{n-1}\sum\limits_{v\in U_0}m_{uf^i(v)}\omega^i x(v)=\sum\limits_{v\in U_0}\sum\limits_{i=0}^{n-1}m_{f^j(u^\prime)f^i(v)}\omega^i x(v)\\
			&=\omega^j\sum\limits_{v\in U_0}\left(\sum\limits_{i=0}^{n-1}\omega^{i-j}m_{u^\prime f^{i-j}(v)} \right)x(v)\\
			&=\omega^j(M^{
				\omega
			}_H(f,U_0) x)(u)=\omega^j\lambda x(u)=\lambda x_{\omega}(u)
		\end{align*}
		For all $u\in X$, since $\omega\ne 1$ is an $n$-th root of unity,
		\begin{align*}
				(M_Hx_{\omega})(u)&=\sum\limits_{v\in V(H)}m_{uv}x_{\omega}(v)=\sum\limits_{i=0}^{n-1}\sum\limits_{v\in U_0}m_{f^i(u)f^i(v)}x_{\omega}(v)\\
				&=\sum\limits_{v\in U_0}m_{uv}\left(\sum\limits_{i=0}^{n-1}\omega^i\right)x(v)=0
		\end{align*}
	\end{proof}
\begin{exm}\rm
    Consider the hypergraph $H$ illustrated in \Cref{fig:rot}.  We have described a rotation $f$ of order $3$ in the hypergraph $H$ in the \Cref{ex-hyp-rot}. Other than $1$ the third-roots of  unity are $\omega_1=\frac{1}{2}(-1+i\sqrt{3})$, and $\omega_2=\frac{1}{2}(-1-i\sqrt{3})$.  Therefore, ${A^{(r)}}^{\omega_1}_H(f,U_0)=\left[\begin{smallmatrix}
        -1 & 2 &  -\omega_1\\ 2 & -1 & -\omega_1\\ 1+\omega_1 &  1+\omega_1 & 0 
    \end{smallmatrix}\right]$. The eigenvalues of this matrix are -3, -1, 2 with eigenvectors  $x^{(11)}:U_0\to\mathbb{C}$, $x^{(12)}:U_0\to\mathbb{C}$, and $x^{(13)}:U_0\to\mathbb{C}$, respectively, and which are defined as
    $$x^{(11)}=\left(\begin{smallmatrix}
        2&3&4\\
        -1&1&0
    \end{smallmatrix}\right), 
    x^{(12)}=\left(\begin{smallmatrix}
        2&3&4\\
       \frac{ \omega_1}{2}&\frac{ \omega_1}{2}&1
    \end{smallmatrix}\right),\text{~and~} x^{(12)}=\left(\begin{smallmatrix}
        2&3&4\\
       -\omega_1&-\omega_1&1
    \end{smallmatrix}\right).$$
    As \Cref{inheritrotation} suggests, $-3, -1, 2$ are eigenvalues of $A^{(r)}_H$ with the corresponding eigenvectors  $x^{(11)}_{\omega_1}$, $x^{(12)}_{\omega_1}$, and $x^{(13)}_{\omega_1}$.

    The other rotation matrix is  ${A^{(r)}}^{\omega_2}_H(f,U_0)=\left[\begin{smallmatrix}
        -1 & 2 &  -\omega_2\\ 2 & -1 & -\omega_2\\ 1+\omega_2 &  1+\omega_2 & 0 
    \end{smallmatrix}\right]$ with its eigenvalues $ -3, -1, 2$ , and the corresponding  eigenvectors are $x^{(21)}:U_0\to\mathbb{C}$, $x^{(22)}:U_0\to\mathbb{C}$, and $x^{(23)}:U_0\to\mathbb{C}$, respectively, and which are defined as 
     $$x^{(21)}=\left(\begin{smallmatrix}
        2&3&4\\
        -1&1&0
    \end{smallmatrix}\right), 
    x^{(22)}=\left(\begin{smallmatrix}
        2&3&4\\
       \frac{ \omega_2}{2}&\frac{ \omega_2}{2}&1
    \end{smallmatrix}\right),\text{~and~} x^{(32)}=\left(\begin{smallmatrix}
        2&3&4\\
       -\omega_2&-\omega_2&1
    \end{smallmatrix}\right).$$
    Consequently, $ -3, -1, 2$ are eigenvalues of $A^{(r)}_H$ with the corresponding eigenvectors  $x^{(11)}_{\omega_2}$, $x^{(12)}_{\omega_2}$, and $x^{(13)}_{\omega_2}$.
\end{exm}

    In the \Cref{inheritrotation}, the condition $\omega\ne 1$ is necessary to show  $(M_Hx_{\omega})(u)=0$ for all $u\in X$. For $\omega=1$, $(M_Hx_{\omega})(u)=0$ might not hold for some $u\in X$, as a consequence an eigenvalue of $M^{1}_H(f,U_0) $ might not be an eigenvalue of $M_H$. For instance, for the hypergraph $H$ considered in \Cref{ex-hyp-mat}, the rotation matrix ${A^{(r)}}^{1}_H(f,U_0)=\left[\begin{smallmatrix}
         2 & 5 & 2\\ 5 & 2 & 2\\ 2 & 2 & 0 
    \end{smallmatrix}\right]$ and its eigenvalues are $-3,-1,8$ but $8$ is not an eigenvalue of $A^{(r)}_H$.  If the invariant set $X=\emptyset$, then we do not require the condition $\omega\ne1$, and in the next result, we show that the complete spectrum of $M_H$ can be described in terms of the spectra of $M^{\omega}_H(f, U_0) $ for all the $n$-th roots of unity.

\begin{cor}\label{xnull-cor}
    Let $H$ be a hypergraph and $f$ be a rotation of order $n$ in $H$ with the $i$-th component $U_i$, such that, $V(H)=\bigcup\limits_{i=0}^{n-1}U_i$, and the invariant set $X=\emptyset$. Suppose that $M_H$ is a $f$-compatible matrix associated with $H$. Each eigenvalue of $M^{\omega}_H(f,U_0) $ is an eigenvalue of $ M_H$ for any $n$-th root of unity $\omega$.
\end{cor}
\begin{proof}
    If $\omega\ne 1$, then it follows from the \Cref{inheritrotation} that each eigenvalue of $M^{\omega}_H(f,U_0) $ is an eigenvalue of $ M_H$. For $M^{1}_H(f,U_0) $ we have the following argument.

    In that case if $M^{1}_H(f,U_0) x=\lambda x$ for some $x:U_0\to\mathbb{C}$, then $x_{\omega}:V(H)\to\mathbb{C}$ is such that $x_{\omega}(v)=x(f^{-i}(v))$ for any $v\in U_i$ for $i=0,1,\ldots,n-1$. Consequently, for any $u\in U_i$, there exists $u'\in U_0$ such that $f^i(u')=u $, and $$(M_Hx_{\omega})(v)=\sum\limits_{v\in V(H)}m_{uv}x_{\omega}(v)=\sum\limits_{j=0}^{n-1}\sum\limits_{v\in U_0}m_{u'f^j(v)}x(v)=(M^{1}_H(f,U_0) x)(u')=\lambda x(u')=\lambda x_{\omega}(u).$$ Thus, $\lambda $ is also an eigenvalue of $M_H$ with  an eigenvector $x_{\omega}$.
    
\end{proof}
  Suppose that $f$ is a rotation of order $n$ in a hypergraph $H$. For each $v\in V(H)$, the \emph{$f$-orbit} of $v$ is $O_f(v)=\{f^i(v):\text{~for each non-negative integer~}i\}$.
	 If $f$ is a rotation of order $n$ in $H$ with $i$-th component of $f$ is $U_i$ then for two distinct $u,v\in U_i$, we have $O_f(u)\cap O_f(v)=\emptyset $. If $X$ is the invariant set under the rotation $f$ then $O_f(v)=\{v\}$ for all $v\in X$. The rotation orbits lead us to the following disjoint partition of $V(H)$.
	$$V(H)=\left(\bigcup\limits_{v\in U_0}O_f(v)\right)\cup \left(\bigcup\limits_{v\in X}O_f(v)\right) $$
	We refer to this partition of $V(H)$ as $f$-orbit partition.
Since each rotation is an automorphism, we have the following Corollary of the \Cref{aut-equit}.
\begin{cor}
		Let $H$ be a hypergraph. Suppose that $f$ is a rotation of order $n$ in $H$. The $f$-orbit partition is an equitable partition for any $f$-compatible matrix $M_H$.
\end{cor}
 The above Corollary and the \Cref{lem-orbit} lead us to the following result. 
 \begin{cor}\label{lem-orbit-rot}
		Let $H$ be a hypergraph. Suppose that $f$
		is a rotation of order $n$ in $H$. For any $f$-compatible matrix $M_H$, each eigenvalue of $[M_H/\mathcal{O}_f] $ is an eigenvalue of $M_H$.
	\end{cor}
\begin{exm}\label{exm-quotient-lem}
    \rm Recall the hypergraph $H$, and associated matrix  $A^{(r)}_H$  illustrated in \Cref{fig:rot}.  The quotient matrix  $[A^{(r)}_H/\mathcal{O}_f]=\left[\begin{smallmatrix}
         0 & 3 & 3 & 0\\ 1 & 2 & 5 & 2\\ 1 & 5 & 2 & 2\\ 0 & 2 & 2 & 0 
    \end{smallmatrix}\right]$ has the eigenvalues $-3$, $ 0$, $ \frac{7}{2}-\frac{\sqrt{105}}{2}$ , and $\frac{\sqrt{105}}{2}+\frac{7}{2}$ with the eigenvectors $y^{(1)}:\mathcal{O}_f\to\mathbb{C}$, $y^{(2)}:\mathcal{O}_f\to\mathbb{C}$, $y^{(3)}:\mathcal{O}_f\to\mathbb{C}$, and $y^{(4)}:\mathcal{O}_f\to\mathbb{C}$, respectively, and which are defined as 
    $ y_1=\left(\begin{smallmatrix}
        O_f(1)&O_f(2)&O_f(3)&O_f(4)\\
        -2&0&0&1
    \end{smallmatrix}\right),
    y_2=\left(\begin{smallmatrix}
        O_f(1)&O_f(2)&O_f(3)&O_f(4)\\
        0&-1&1&0
    \end{smallmatrix}\right),
     y_3=\left(\begin{smallmatrix}
        O_f(1)&O_f(2)&O_f(3)&O_f(4)\\
        \frac{3}{2}&\frac{7}{8}-\frac{\sqrt{105}}{8}&\frac{7}{8}-\frac{\sqrt{105}}{8}&1
    \end{smallmatrix}\right), \text{~and~} y_4=\left(\begin{smallmatrix}
        O_f(1)&O_f(2)&O_f(3)&O_f(4)\\
        \frac{3}{2}&\frac{7}{8}+\frac{\sqrt{105}}{8}&\frac{7}{8}+\frac{\sqrt{105}}{8}&1
    \end{smallmatrix}\right)$. By \Cref{lem-orbit-rot}, $-3$, $ 0$, $ \frac{7}{2}-\frac{\sqrt{105}}{2}$ , and $\frac{\sqrt{105}}{2}+\frac{7}{2}$ are eigenvalues of  $A^{(r)}_H$ with eigenvectors ${y^{(1)}}_{\mathcal{O}_f}$,  ${y^{(2)}}_{\mathcal{O}_f}$, ${y^{(3)}}_{\mathcal{O}_f}$, and ${y^{(4)}}_{\mathcal{O}_f}$, respectively.
\end{exm}

 The following result shows that the eigenvalues of $M_H$ can be computed from the eigenvalues of smaller matrices.
	\begin{thm}\label{rotation-final}
		Let $H$ be a hypergraph. Suppose that $f$ is a rotation of order $n$ with $0$-th component $U_0$, and $M_H$ is a $f$-compatible matrix associated with $H$. The complete spectrum of $M_H$ are the eigenvalues of  $n-1$ matrices  $M^{\omega}_H(f,U_0) $ of order $|U_0|$ for all $n$-th root of unity $\omega\ne 1$, and the eigenvalues of the matrix $[M_H/\mathcal{O}_f]$ of order $|U_0|+|X|$.
	\end{thm}
        \begin{proof}
            If $H$ is a hypergraph and $f$ is a rotation of order $n$ with $0$-th component $U_0$, then $|V(H)|=n|U_0|+|X|$. Let $\omega_1,\omega_2,\ldots,\omega_{n-1}$ be  the roots of $x^n=1$ other than $1$.
  Suppose that $|U_0|=m$ and $\{y^{(ij)}:j=1,2,\ldots,m\}$ are the linearly  independent eigenvectors of the $m\times m$ matrix $M^{\omega_i}_H(f,U_0) $ for all $i=1,2,\ldots,n-1$. Let $|X|=p$, and $\{z^{(i)}:i=1,2,\ldots,m+p\}$ are the   linearly  independent eigenvectors of the $(m+p)\times (m+p)$ matrix $[M_H/\mathcal{O}_f]$. Suppose that  $\mathcal{Y}_i=\{y^{(ij)}_{\omega_i}:j=1,2,\ldots,m\}$, and $\mathcal{Y}=\bigcup\limits_{i=1}^{n-1}\mathcal{Y}_i=\{y^{(ij)}_{\omega_i}:j=1,2,\ldots,m;i=1,2,\ldots,n-1\}$, and $\mathcal{Z}=\{z^{(i)}_{O_f}:i=1,2,\ldots,m+p\}$. It is enough to show that $\mathcal{Y}\cup\mathcal{Z}$ are all the  linearly  independent eigenvectors of $M_H$. By \Cref{inheritrotation}, and \Cref{lem-orbit-rot}, it is evident that all the elements of $\mathcal{Y}\cup\mathcal{Z}$ are eigenvectors of $M_H$. Since the cardinality of $\mathcal{Y}\cup\mathcal{Z}$ is $nm+p$, that is the order of $M_H$, it is enough to show the collection $\mathcal{Y}\cup\mathcal{Z}$ is linearly independent. 
Clearly,  $\{y^{(ij)}:j=1,2,\ldots,m\}$ is a linearly independent set of eigenvectors. Consequently, $\mathcal{Y}_i=\{y^{(ij)}_{\omega_i}:j=1,2,\ldots,m\}$ is also linearly independent for all $i=1,2,\ldots,n-1$. For two distinct $p,q\in\{1,2,\ldots,n-1\}$, if $y^{(pr)}_{\omega_p}\in \mathcal{Y}_p$, $y^{(qs)}_{\omega_q}\in \mathcal{Y}_q$, and $\langle\cdot,\cdot\rangle$ is the usual inner product on $\mathbb{C}^n$, then
  \begin{align*}
      \langle y^{(pr)}_{\omega_p},y^{(qs)}_{\omega_q}\rangle=\sum\limits_{v\in V(H)}y^{(pr)}_{\omega_p}(v)\overline{y^{(qs)}_{\omega_q}(v)}=\sum\limits_{u\in U_0}y^{(pr)}(u)\overline{y^{(qs)}(u)}\sum\limits_{i=0}^{n-1}(\omega_p\overline{\omega_q})^i=0.
  \end{align*}
  Therefore, $\mathcal{Y}=\bigcup\limits_{i=1}^{n-1}\mathcal{Y}_i=\{y^{(ij)}_{\omega_i}:j=1,2,\ldots,m;i=1,2,\ldots,n-1\}$ is linearly independent set of eigenvectors.   The collection $\mathcal{Z}$ is linearly independent. The fact $1+\omega+\ldots+\omega^{n-1}=0$ for any root $\omega(\ne 1)$ of the equation $x^n=1$, establishes that any element of $\mathcal{Y}$ is orthogonal to each element of $\mathcal{Z}$ in the usual inner product space. Therefore, $\mathcal{Y}\cup\mathcal{Z}$ is a collection of  linearly independent eigenvectors of $M_H$.
        \end{proof}

Let $f$ be a rotation of order $n$ in a hypergraph $H$ and $|O_f(v)|$ be the cardinality of the orbit of a vertex $v\in V(H)$. 
If $U_i$ is the $i$-th component of the rotation $f$, for all $i=0,1,\ldots,n-1$, and $X$ is the invariant set under $f$, then the cardinality of the orbit of $v$, $|O_f(v)|=n$ for any $v\in U_i$, for any $i=0,\ldots,n-1$, and $|O_f(v)|=1$ if $v\in X$. 
If the invariant set is empty, then $[M_H/\mathcal{O}_f]=M^{1}_H(f,U_0)$ and then the complete list of eigenvalues of $M_H$ is the union of the eigenvalues of the collection of matrices $\{M^{\omega}_H(f,U_0):\omega\text{~is a root of ~}x^n-1=0\}$. In \cite{equit_benjamin1}, a graph automorphism with orbit sizes equal to either $n>1$ or $1$ is referred to as \emph{basic automorphisms}, and a graph automorphism, with all its orbits of the same size, is referred to as \emph{uniform automorphism}. Thus, any rotation in a hypergraph is a basic automorphism of the hypergraph, and a rotation with an empty invariant set is a uniform automorphism.

For a natural number $n>1$, $\{\omega_0,\omega_1,\ldots,\omega_{n-1}\}$ is the collection of all the roots of $x^n-1=0$, where $\omega_k=e^{\iota\frac{2k\pi}{n}}$, for $k=0,1,\ldots,n-1$. Given a rotation $f$ of order $n$ in a hypergraph $H$, if an $f$-compatible matrix is symmetric then the equality $ \omega_j^{-1}=\omega_{n-j}$ leads us to
$$\sum\limits_{i=0}^{n-1}\omega^i_jm_{uf^i(v)}=\sum\limits_{i=0}^{n-1}\omega^i_{n-j}m_{vf^i(u)}.   $$
Therefore, the pair of rotation matrices $M^{\omega_j}_H(f,U_0) $, and $M^{\omega_{n-j}}_H(f,U_0) $ have the same eigenvalues. Consequently, the following Corollary of the \Cref{rotation-final} follows.
\begin{cor}\label{cor-even-odd}
    Let $H$ be a hypergraph, and $f$ be a rotation of order $n$ in $H$. Given any symmetric $f$-compatible matrix $M_H$, if $n$ is odd, then there can be at most $|U_0|+|X|$ many simple eigenvalues of $M_H$, otherwise, if $n$ is even, then there can be $2|U_0|+|X|$ number of simple eigenvalues, where $U_0$, and $X$ are, respectively, the $0$-th component, and the invariant set of the rotation $f$.
\end{cor}
\begin{proof}
    Since $M_H$ is symmetric, the pair of rotation matrices $M^{\omega_j}_H(f,U_0) $, and $M^{\omega_{n-j}}_H(f,U_0) $ have the same eigenvalues for $j=1,\ldots,n-1$. If $n$ is even then $M^{\omega_j}_H(f,U_0) =M^{\omega_{n-j}}_H(f,U_0) $ for $j=\frac{n}{2}$, the result follows.
\end{proof}
\subsection{Rotation decomposition of hypergraph automorphisms} A rotation is an automorphism, but every automorphism is not a rotation. Thus, if a matrix $M_H$ is compatible with an automorphism $f$, which is not a rotation, we can not apply the \Cref{rotation-final}. Now, we explore the relation between an automorphism $f$ and the eigenvalues of any $f$-compatible matrix. 
A hypergraph $H$ with a rotation $f$ of order $n$ induces a new hypergraph $H_f$ such that $V(H_f)=\mathcal{O}_f$, the collection of all $f$-orbits in $H$ and $E(H_f)=\{\tilde e=\{O_f(v):v\in e\}: e\in E(H)\}$, where $\tilde e=\{O_f(v):v\in e\}$ for all $e\in H$.
Since $[M_H/\mathcal{O}_f]$ is a matrix whose row and column are indexed by the vertex set of the hypergraph $H_f$. Consequently, if $g$ is a rotation in $H_f$, and $[M_H/\mathcal{O}_f]$ is a $g$-compatible matrix then we can use \Cref{inheritrotation} again to find the eigenvalues of $[M_H/\mathcal{O}_f]$.
\begin{df}
    Let $H$ be a hypergraph and $f$ be a rotation of order $n$ in $H$. If $U_i$ is the $i$-th component of the rotation $f$, then we refer to $D_f=\bigcup\limits_{i=0}^{n-1}U_i$ as the \emph{active domain} of the rotation $f$.
    Two hypergraph rotations $f_1$, and $f_2$ of order $n_1$, and $n_2$, respectively, in a hypergraph $H$ are called mutually \emph{disjoint} if the two active domains of $f_1$, and $f_2$ are disjoint, that is, $D_{f_1}\cap D_{f_2}=\emptyset$.
\end{df}
 Given any automorphism $f$ of a hypergraph $H$, $f$ is a permutation of $V(H)$. Thus, $f$ can be expressed as a product of disjoint cycles. The next result shows that each automorphism $f$ can be expressed as the product of disjoint rotations.
 \begin{thm}\label{aut-decomp-rot}
     Let $H$ be a hypergraph. Any automorphism of $H$ can be expressed as the product of disjoint rotations.
 \end{thm}
 \begin{proof}
     Let $f$ be an automorphism of the hypergraph $H$. Since each permutation can be written as a product of disjoint cycles,
     $f=(f_{11}\circ f_{12}\circ\ldots\circ f_{1n_1})\circ\ldots\circ(f_{k1}\circ f_{k2}\circ\ldots\circ f_{kn_k})$, where for each $i=1,\ldots,k$, the length of the cycle $f_{ij}$ is $l_i$ for all $j=1,2,\ldots,n_i$, and $l_{i_1}\ne l_{i_2}$ for $i_1\ne i_2$. 
     Without loss of generality, we assume $l_{ki}=1$ for all $i=1,2,\ldots,n_k$. If no cycle of length $1$ exists in $f$, then $n_k=0$. Since the cycles are disjoint, the vertex set $V(H)$ can be written as the following disjoint partition
      $V(H)=\bigcup\limits_{i=1}^k\left(\bigcup\limits_{j=1}^{n_i}U_{ij}\right)$ such that $f_{ij}(v)=v$ for all $V(H)\setminus U_{ij}$ for all $i=1,\ldots,k$, and $j=1,\ldots,n_k$.
      For each $i=1,\ldots,k$, consider the map $f_i=f_{i1}\circ f_{i2}\circ\ldots\circ f_{in_i}$.
      Pick one element $u_j$ from each $U_{ij}$ for all $j=1,\ldots,n_i$, and assume $U_0^{(i)}=\{u_1,u_2,\ldots,u_{n_i}\}$. Here $f_i:V(H)\to V(H)$ is a rotation of order $l_i$ with the $0$-th component $U_0^{(i)}$. The active domain of $f_i$ is $\bigcup\limits_{j=1}^{n_i}U_{ij}$. Consequently, the invariant set of $f_i$ is $V(H)\setminus \bigcup\limits_{j=1}^{n_i}U_{ij}$. Since for two distinct $i_1,i_2\in\{1,\ldots,k\}$, the two active domains $\bigcup\limits_{j=1}^{n_{i_1}}U_{i_1j}$, and $\bigcup\limits_{j=1}^{n_{i_2}}U_{i_2j}$ are disjoint, two rotations $f_{i_1}$, and $f_{i_2}$ are disjoint. Thus, $f=f_{1}\circ f_{2}\circ\ldots\circ f_{k}$ is product of disjoint rotations.
 \end{proof}
For any automorphism $f$ in the group $Aut(H)$, it can be expressed as a composition of disjoint rotations $f_1, f_2, \ldots, f_k$, for some $1\le k\le|V(H)|$, each having orders $l_1, l_2, \ldots, l_k$ respectively. We refer $f=f_{1}\circ f_{2}\circ\ldots\circ f_{k}$ to as the \emph{rotation decomposition} of the automorphism $f$. Let $\omega_{ij} = e^{\iota{2j\pi}/{l_i}}$ for $j=0,1,\ldots,l_i-1$ and $i=1,2,\ldots,k$. In other words, $\omega_{i0}(=1), \omega_{i1}, \ldots, \omega_{il_{i-1}}$ represent the roots of $x^{l_i}-1=0$ for all $i=1,\ldots,k$. Given the disjoint nature of $f_1, \ldots, f_k$, the vertex set $V(H)$ can be partitioned into disjoint subsets: $V(H) = V_1 \cup \ldots \cup V_k$, where $V_i$ is the active domain of rotation $f_i$ for all $i=1,\ldots,k$. If $l_i$ is the order of rotation $f_i$, then $V_i = U_0^{(i)} \cup \ldots \cup U_{l_i-1}^{(i)}$, where $U_j^{(i)}$ is the $j$-th component of $f_i$ for $j=0,1,\ldots,l_1-1$. If $|U_0^{(i)}|=m_i$, then $|V_i|=l_i m_i$, and $|V(H)|=\sum_{i=1}^k l_i m_i$. Utilizing the result in \Cref{inheritrotation}, it follows that each eigenvalue of the $m_i \times m_i$ matrix $M_H^{\omega_{ij}}(f_i,U_0^{(i)})$ is an eigenvalue of $M_H$ for all $j=1,\ldots,l_i-1$, $i=1,\ldots,k$. In total, these rotation matrices contribute $\sum_{i=1}^k(l_i-1)m_i$ eigenvalues to $M_H$. For the remaining eigenvalues, we employ \Cref{lem-orbit}. The complete list of eigenvalues is encapsulated in the following result.

\begin{thm}\label{aut-comp}
    Let $H$ be a hypergraph and $f\in Aut(H)$. If $f=f_{1}\circ f_{2}\circ\ldots\circ f_{k}$ is a rotation-decomposition of $f$ where $f_i$ is a rotation of order $l_i$ with its $0$-th component $U_0^{(i)}$ for all $i=1,\ldots,k$, then for any $f$-compatible matrix $M_H$, the complete list of eigenvalues are the
    eigenvalues of the matrix $M_H^{\omega_{ij}}(f_i, U_0^{(i)})$ for all $j=1,\ldots,l_i-1$, $i=1,\ldots,k$, and the eigenvalues of the quotient matrix $[M_H/\mathcal{O}_f]$.
\end{thm}
\begin{proof}Being an $m_i\times m_i$ matrix, $M_H^{\omega_{ij}}(f_i,U_0^{(i)})$ has $m_i$ numbers of eigenvalues, where $m_i$ is the cardinality of the $0$-th component of $U_0^{(i)}$. Thus by the \Cref{inheritrotation}, for all $j=1,\ldots,l_i-1$, $i=1,\ldots,k$, we have total $\sum_{i=1}^k(l_i-1)m_i$ eigenvalues of $M_H$, where $l_i$ is the order of $f_i$. By the \Cref{lem-orbit}, $[M_H/\mathcal{O}_f]$ provides $\sum_{i=1}^km_i$ many eigenvalues of $M_H$. Since $\sum_{i=1}^k(l_i-1)m_i+\sum_{i=1}^km_i=|V(H)|$, it is enough to show the corresponding eigenvectors are linearly independent.
    Suppose that $V_i$ is the active domain of the rotation $f_i$ for all $i=1,\ldots,k$. For two distinct $p,q\in\{1,\ldots,k\}$, since $V_p\cap V_q=\emptyset$, if $x:U_0^{(p)}\to\mathbb{C}$, and $y:U_0^{(q)}\to\mathbb{C}$ are eigenvectors of $M_H^{\omega_{pj}}(f_p,U_0^{(p)})$ and $M_H^{\omega_{qj}}(f_q,U_0^{(q)})$, respectively, then the inner product $\langle x_{\omega_p},y_{\omega_q}\rangle=0 $. 
    For any eigenvector $z$ of $[M_H/\mathcal{O}_f]$, if $x$ is an eigenvector of $M_H^{\omega_{ij}}(f_i,U_0^{(q)})$ for any $j=1,\ldots,l_i-1$, $i=1,\ldots,k$, then $\langle x_{\omega_i},z_{\mathcal{O}_f}\rangle=0$. This completes the proof.
\end{proof}
 The \Cref{aut-comp} provides a decomposition of an automorphism-compatible matrix into smaller matrices such that the union of the eigenvalues of the smaller matrices is the complete list of eigenvalues of the automorphism-compatible matrix. In the proof of \Cref{aut-decomp-rot}, we club all the cycles of length $l_i$ to construct the rotation $f_i$ of order $l_i$. Thus, in the rotation decomposition, the rotations have distinct orders. Being disjoint, these rotations $f_1,f_2,\ldots,f_k$ commutes, and hence $f^m=f_1^m\circ\ldots\circ f_k^m$. If $l$ is the least common multiple of $l_1,\ldots,l_k$, then $l$ is the order of $f$. If $l_1,\ldots,l_k$ are all primes numbers, then the order of $f$ is a product of distinct primes. In \cite{benjamin2}, a graph automorphism with its order equal to a product of distinct primes is referred to as \emph{separable automorphism}. This paper uses basic automorphisms to provide the eigenvalues and eigenvectors of a matrix compatible with separable graph automorphisms from the smaller matrices. Later in \cite{Benjamin-webb-gen-equitable-2019}, uniform, basic, and separable graph automorphisms are used to deal with any general graph automorphism and a method is provided to find the eigenvalues of a matrix compatible with any graph automorphisms using smaller matrices associated with some basic automorphisms.

For any $n$-th root of unity $\omega$, recall the definition of the rotation matrix $M^{\omega}_H(f,U_0)=\left(r_{uv}\right)_{u,v\in U_0}$, indexed by $U_0$, where
		$r_{uv}=\sum\limits_{i=0}^{n-1}\omega^im_{uf^i(v)}
		$. Since $|r_{uv}|\le \sum\limits_{i=0}^{n-1}|\omega^i||m_{uf^i(v)}|=\sum\limits_{i=0}^{n-1}m_{uf^i(v)}$,
  therefore, the $(u,v)$-th entry of $M^{\omega}_H(f,U_0)$ can be at most the $(u,v)$-th entry of $M^{1}_H(f,U_0)$. Now, we show that the \emph{spectral radius} of $M_H$, that is, $\rho(M_H)=\max\{|\lambda|:\lambda \text{~is an eigenvalue of $M_H$}\}$ is the spectral radius of $[M_H/\mathcal{O}_f]$.

  \begin{thm}
      Let $H$ be a hypergraph, and $f\in Aut(H)$. If $M_H$ is an $f$-compatible matrix, then $\rho(M_H)=\rho([M_H/\mathcal{O}_f])$.
  \end{thm}
  \begin{proof}
      Let $f=f_{1}\circ f_{2}\circ\ldots\circ f_{k}$ be the {rotation decomposition} of the automorphism $f$. The collection of all the $f$-orbit, $\mathcal{O}_f$ can be partitioned into disjoint set as $\mathcal{O}_f=\mathcal{O}_1\cup\ldots\cup \mathcal{O}_k$, where  $\mathcal{O}_i=\{O_f(v):v\in V_i\}$, where $V_i$ is the active domain of the rotation $f_i$ for all $i=1,\ldots,k$. Since all the rotations in the rotation decomposition are disjoint, we have $V_i\cap V_j=\emptyset$ for two distinct $i,j=1,\ldots,k$. Thus, $\mathcal{O}_i\cap \mathcal{O}_j=\emptyset$. Consider the submatrix $[M_H/f_i]=\left(b_{pq}\right)_{p,q\in \mathcal{O}_i}$ of $[M_H/\mathcal{O}_f] =\left(b_{pq}\right)_{p,q\in\mathcal{O}_f}$. Here $\sum\limits_{w\in q}m_{uw}=b_{pq}$ for all $p,q\in \mathcal O_f $, and $u\in p$. For any $p\in \mathcal{O}_i$, the orbit $p=O_{f_i}(v)=\{v,f_i(v),f_i^2(v),\ldots,f_i^{l_i-1}(v)\}$ for some $v\in U^{(i)}_0$. Therefore, if $p,q\in \mathcal{O}_i$, then $p=O_{f_i}(u)$, and $q=O_{f_i}(v)$ for some $u,v\in U_0^{(i)}$. Thus, $b_{pq}=\sum\limits_{k=0}^{l_i-1}m_{uf_i^k(v)}=$ the $(u,v)$-th entry of
      $M^{1}_H(f_i,U_0^{(i)})$. Therefore, $[M_H/f_i]=M^{1}_H(f_i,U_0^{(i)})$. That is, $M^{1}_H(f_i,U_0^{(i)})$ is a submatrix of $M_H$.

      Let $\omega_{ij} = e^{\iota\frac{2j\pi}{l_i}}$ for $j=0,1,\ldots,l_i-1$ and $i=1,2,\ldots,k$, where $l_i$ is the order of the rotation $f_i$. Since the rotation matrix is $M^{\omega_{ij}}_H(f_i,U_0^{(i)})=\left(r_{uv}\right)_{u,v\in U_0}$, and  $|r_{uv}|\le \sum\limits_{i=0}^{n-1}|\omega^i||m_{uf^i(v)}|=\sum\limits_{i=0}^{n-1}m_{uf^i(v)}$, therefore, the 
$(u,v)$
-th entry of 
$M^{\omega_{ij}}_H(f_i,U_0^{(i)})$
 can be, at most, the 
$(u,v)$
-th entry of $M^{1}_H(f_i,U_0^{(i)})=[M_H/f_i]$, a submatrix of $[M_H/\mathcal{O}]_f$. Thus, by \cite[Corollary 8.1.19 and Corollary 8.1.20]{horn2012matrix}, we have $\rho(M^{\omega_{ij}}_H(f_i,U_0^{(i)}) )\le\rho([M_H/\mathcal{O}]_f)$. Thus, by the \Cref{aut-comp} $\rho(M_H)=\rho([M_H/\mathcal{O}_f])$.
  \end{proof}
 \section{Unit-based symmetries}
In addition to hypergraph automorphisms, certain symmetries inherent in hypergraphs can be explained using equivalence relations. In a hypergraph automorphism $f$, the $f$-orbits correspond to the equivalence classes of a designated equivalence relation. Notably, other equivalence relations defined on the vertex set $V(H)$ of a hypergraph $H$ also represent hypergraph symmetries. For instance, consider the following equivalence relation:
 $$\mathcal{R}_{{H}}=\{\{u,v\}\in V(H)\times V(H):E_u(H)=E_v(H)\}.$$
 We have introduced the notion of the \emph{unit} in a hypergraph in \cite{unit}. Any $\mathcal{R}_{{H}}$-equivalence class is called a unit. Each unit  corresponds to an $E_0\subseteq E(H)$ such that for each vertex $v$ belongs to the $\mathcal{R}_{{H}}$-equivalence class, $E_v(H)=E_0$. We refer to $E_0$ as the \emph{generating set} of the unit. We denote the unit with generating set $E_0$ as $W_{E_0}$. The units in hypergraph $H$ collectively form a partition of the vertex set $V(H)$, and this collection of units is denoted as $\mathfrak{U}(H)$, that is, $V(H)=\bigcup\limits_{W_E\in\mathfrak{U}(H)}W_E$.

	\begin{exm}[]\label{hyp-fig}
		Consider the hypergraph $H$ in \Cref{fig:unit}, with $V(H)=\{n\in \mathbb{N}:n\le 18\}$ and 
		$E(H)=\{e_1,e_2,e_3,e_4,e_5,e_6,e_7\}$, where  $e_1=\{1,2,3,4,9,10\}$, $e_2=\{1,2,3,4,7,8\}$, $e_3=\{1,2,3,4,5,6,15\}$, $e_4=\{1,2,11,12,16\}$, $e_5=\{1,2,13,14\}$, $e_6=\{11,12,16,17,18\}$, and $e_7=\{13,14,17,18\} $.
		The generating sets are $E_1=\{e_1,e_2,e_3,e_4,e_5\}$, $E_2=\{e_1,e_2,e_3\}$, $E_3=\{e_3\}$, $E_4=\{e_2\}$, $E_5=\{e_1\}$, $E_6=\{e_4,e_6\}$, $E_7=\{e_5,e_7\}$ and $E_8=\{e_6,e_7\}$.
		The Corresponding units of $H $ are $W_{E_1}=\{1,2\}$, $W_{E_2}=\{3,4\}$, $W_{E_3}=\{5,6,15\}$, $W_{E_4}=\{7,8\}$, $W_{E_5}=\{9,10\}$, $W_{E_6}=\{11,12,16\}$, $W_{E_7}=\{13,14\}$, and $W_{E_8}=\{17,18\}$.
		
	\end{exm}
	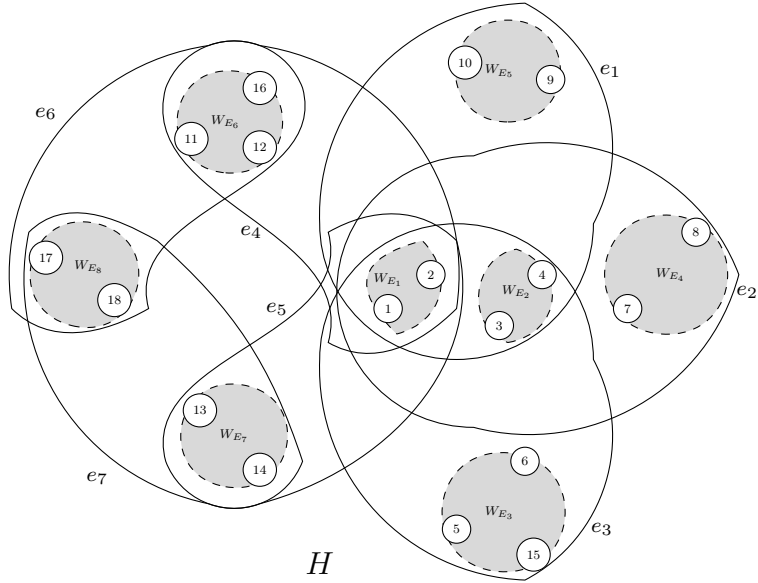
\begin{figure}[ht]
		\centering
		\begin{tikzpicture}[scale=0.45]

		\node [style=none] (0) at (2, 6.5) {};
		\node [style=none] (1) at (0, -4) {};
		\node [style=none] (2) at (4, 0) {};
		\node [style=none] (3) at (-4, 0) {};
		\node [style=none] (4) at (0, 0) {};
		\node [style=none] (5) at (2, -10.5) {};
		\node [style=none] (6) at (4, -4) {};
		\node [style=none] (7) at (-4, -4) {};
		\node [style=none] (8) at (0.5, 2) {};
		\node [style=none] (9) at (0.5, -6) {};
		\node [style=none] (10) at (8.25, -1.5) {};
		\node [style=none] (11) at (-3.5, -2) {};
		\node [style=none] (12) at (-5.75, 5.25) {};
		\node [style=none] (13) at (-3.75, -3.5) {};
		\node [style=none] (14) at (0, -2.5) {};
		\node [style=none] (15) at (-8.5, 4) {};
		\node [style=none] (16) at (-7.25, 5.25) {};
		\node [style=none] (17) at (-9, -2.5) {};
		\node [style=none] (18) at (-13, -2.5) {};
		\node [style=none] (19) at (-4.5, 4) {};
		\node [style=none] (20) at (-7.25, -8.25) {};
		\node [style=none] (21) at (-8.75, -0.5) {};
		\node [style=none] (22) at (-12.5, -0.25) {};
		\node [style=none] (23) at (-4.5, -7) {};
		\node [style=none] (24) at (-5.75, -8.25) {};
		\node [style=none] (25) at (-3.75, -0.25) {};
		\node [style=none] (26) at (0, -0.5) {};
		\node [style=none] (27) at (-8.5, -7) {};
		\node [style=none] (28) at (6.25, 0.25) {};
		\node [style=none] (29) at (6, -3.25) {};
		\node [style=none] (30) at (1.5, 6) {};
		\node [style=none] (31) at (1.5, 3) {};
		\node [style=none] (32) at (1.5, -6.75) {};
		\node [style=none] (33) at (1.25, -10.25) {};
		\node [style=none] (34) at (1.75, -0.75) {};
		\node [style=none] (35) at (1.5, -3.5) {};
		\node [style=none] (36) at (-1, -0.5) {};
		\node [style=none] (37) at (-1.75, -3.25) {};
		\node [style=none] (38) at (-6.5, 4.5) {};
		\node [style=none] (39) at (-6.75, 1.5) {};
		\node [style=none] (40) at (-6.25, -4.75) {};
		\node [style=none] (41) at (-6.75, -7.75) {};
		\node [style=none] (42) at (-10.5, 0) {};
		\node [style=none] (43) at (-11.25, -3) {};
	
		\draw (2.center)[]
			 to [bend left=45] (1.center)
			 to [bend left=45] (3.center)
			 to [bend left=45] (0.center)
			 to [bend left=45] cycle;
		\draw (6.center)
			 to [bend left=45] (5.center)
			 to [bend left=45] (7.center)
			 to [bend left=45] (4.center)
			 to [bend left=45] cycle;
		\draw (10.center)
			 to [bend left=45] (9.center)
			 to [bend left=45] (11.center)
			 to [bend left=45] (8.center)
			 to [bend left=45] cycle;
		\draw (14.center)
			 to [in=-30, out=225] (13.center)
			 to [in=255, out=75] (15.center)
			 to [bend left=45] (12.center)
			 to [bend left=45] cycle;
		\draw (19.center)
			 to [bend right=45] (16.center)
			 to [bend right=45] (18.center)
			 to [in=-150, out=-45] (17.center)
			 to [in=-75, out=105] cycle;
		\draw (23.center)
			 to [bend left=45] (20.center)
			 to [bend left=45] (22.center)
			 to [in=150, out=45] (21.center)
			 to [bend left=15] cycle;
		\draw (25.center)
			 to [in=-255, out=-75] (27.center)
			 to [bend right=45] (24.center)
			 to [in=285, out=15] (26.center)
			 to [in=30, out=-225] cycle;
		\draw [fill=gray!30!white,style=new edge style 1] (29.center)
			 to [bend right=90, looseness=1.75] (28.center)
			 to [bend left=270, looseness=1.75] cycle;
		\draw [fill=gray!30!white,style=new edge style 1] (31.center)
			 to [bend right=90, looseness=1.75] (30.center)
			 to [bend left=270, looseness=1.75] cycle;
		\draw [fill=gray!30!white,style=new edge style 1] (33.center)
			 to [bend right=90, looseness=1.75] (32.center)
			 to [bend left=270, looseness=1.75] cycle;
		\draw [fill=gray!30!white,style=new edge style 1] (35.center)
			 to [bend right=75, looseness=1.50] (34.center)
			 to [bend right=75, looseness=1.25] cycle;
		\draw [fill=gray!30!white,style=new edge style 1] (37.center)
			 to [bend right=60, looseness=1.25] (36.center)
			 to [bend right=60, looseness=1.75] cycle;
		\draw [fill=gray!30!white,style=new edge style 1] (39.center)
			 to [bend right=90, looseness=1.75] (38.center)
			 to [bend left=270, looseness=1.75] cycle;
		\draw [fill=gray!30!white,style=new edge style 1] (41.center)
			 to [bend right=90, looseness=1.75] (40.center)
			 to [bend left=270, looseness=1.75] cycle;
		\draw [fill=gray!30!white,style=new edge style 1] (43.center)
			 to [bend right=90, looseness=1.75] (42.center)
			 to [bend left=270, looseness=1.75] cycle;

 	\node [style=new style 0,scale=0.6] (44) at (-2, -2.5) {\tiny 1};
		\node [style=new style 0,scale=0.6] (45) at (-0.75, -1.5) {\tiny 2};
		\node [style=new style 0,scale=0.6] (46) at (1.25, -3) {\tiny 3};
		\node [style=new style 0,scale=0.6] (47) at (2.5, -1.5) {\tiny 4};
		\node [style=new style 0,scale=0.6] (48) at (0, -9) {\tiny 5};
		\node [style=new style 0,scale=0.6] (49) at (2, -7) {\tiny 6};
		\node [style=new style 0,scale=0.6] (50) at (2.25, -9.75) {\tiny 15};
		\node [style=new style 0,scale=0.6] (51) at (5, -2.5) {\tiny 7};
		\node [style=new style 0,scale=0.6] (52) at (7, -0.25) {\tiny 8};
		\node [style=new style 0,scale=0.6] (53) at (0.25, 4.75) {\tiny 10};
		\node [style=new style 0,scale=0.6] (54) at (2.75, 4.25) {\tiny 9};
		\node [style=new style 0,scale=0.6] (55) at (-5.75, 2.25) {\tiny 12};
		\node [style=new style 0,scale=0.6] (56) at (-7.75, 2.5) {\tiny 11};
		\node [style=new style 0,scale=0.6] (57) at (-5.75, 4) {\tiny 16};
		\node [style=new style 0,scale=0.6] (58) at (-7.5, -5.5) {\tiny 13};
		\node [style=new style 0,scale=0.6] (59) at (-5.75, -7.25) {\tiny 14};
		\node [style=new style 0,scale=0.6] (60) at (-12, -1) {\tiny 17};
		\node [style=new style 0,scale=0.6] (61) at (-10, -2.25) {\tiny 18};
		\node [style=none] (62) at (4.5, 4.5) {\tiny $e_1$};
		\node [style=none] (63) at (8.5, -2) {\tiny $e_2$};
		\node [style=none] (64) at (4.25, -9) {\tiny $e_3$};
		\node [style=none] (65) at (-5.25, -2.5) {\tiny $e_5$};
		\node [style=none] (66) at (-6, -0.25) {\tiny $e_4$};
		\node [style=none] (67) at (-10.5, -7.5) {\tiny $e_7$};
		\node [style=none] (68) at (-12, 3.25) {\tiny $e_6$};
		\node [style=none,scale=0.6] (69) at (6.25, -1.5) {\tiny $W_{E_4}$};
		\node [style=none,scale=0.6] (70) at (1.25, 4.5) {\tiny $W_{E_5}$};
		\node [style=none,scale=0.6] (71) at (1.25, -8.5) {\tiny $W_{E_3}$};
		\node [style=none,scale=0.6] (72) at (1.75, -2) {\tiny $W_{E_2}$};
		\node [style=none,scale=0.6] (73) at (-2, -1.65) {\tiny $W_{E_1}$};
		\node [style=none,scale=0.6] (74) at (-6.5, -6.25) {\tiny $W_{E_7}$};
		\node [style=none,scale=0.6] (75) at (-6.75, 3) {\tiny $W_{E_6}$};
		\node [style=none,scale=0.6] (76) at (-10.75, -1.25) {\tiny $W_{E_8}$};
		\node [style=none] (77) at (-4, -10) {$H$};
	
\end{tikzpicture}

			\caption{Hypergraph with its unit}
			\label{fig:unit}
		\end{figure}
		
		\subsection{Unit-compatible Matrices}Unit-compatible matrices are the matrices associated with a hypergraph that encodes the symmetry of units in the hypergraph.	
		\begin{df}[Unit-compatible Matrices] Let $H$ be a hypergraph. 
			A matrix $M_H=\left(m_{uv}\right)_{u,v\in V(H)}$ is called \emph{unit-compatible matrices} if for all $W_E\in\mathfrak{U}(H)$, and any pair of vertices  $u,v\in W_E$,
				 $m_{uw}=m_{vw}$, $m_{wu}=m_{wv}$ for all $w\in V(H)\setminus\{u,v\}$,
			 $m_{uu}=m_{vv}=d_{M_H}(W_E)$, and $m_{uv}=m_{vu}$.
		\end{df}
		Thus, if $M_H$ is a unit-compatible matrix associated with a hypergraph $H$, and $W_E\in\mathfrak{U}(H)$ then $m_{uu}=d_{M_H}(W_E)$, a constant for all $u\in W_E$. Moreover, $m_{uw}=s^w_{W_E}(M_H)$ for all $w\in V(H)\setminus W_E$. For $u,v,u^\prime,v^\prime\in W_E$, we have
		$m_{u^\prime v^\prime}=m_{u^\prime v}=m_{ vu^\prime}=m_{vu}=m_{uv}.$
		Therefore, each $W_E\in\mathfrak{U}(H)$ corresponds to a constant $r_{M_H}(W_E)$, such that, $m_{uv}=r_{M_H}(W_E)$, for all $u,v\in W_E$. Before going into the next result, note that for any $U\subseteq V(H)$, the characteristic function of $U$, $\chi_{U}:V(H)\to\{0,1\}$ is such that $\chi_U(v)=1$ if $v\in U$, and otherwise $\chi_U(v)=0$.

  	\begin{prop}\label{lem-unit}
			Let $H$ be a hypergraph. For any $W_{E}\in\mathfrak{U}(H)$ with $|W_E|>1$, any unit-compatible matrix $M_H=\left(m_{uv}\right)_{u,v\in V(H)}$ has an eigenvalue $d_{M_H}(W_E)-r_{M_H}(W_E)$ with multiplicity at least $|W_{E}|-1$.
		\end{prop}
		\begin{proof}
			Let $ W_{E}=\{v_0,v_1,\ldots,v_n\}$, and $y_i=\chi_{\{v_i\}}-\chi_{\{v_0\}}$ for all $i=1,2,\ldots,n$. For any $v\in V(H)$ we have,
			$M_H(y_j)(u)=\sum\limits_{v\in V(H)}m_{uv}y_j(v).$  If $u\notin\{v_0,v_i\}$ then $M_H(y_j)(u)=m_{uv_i}-m_{uv_0}=0$. If $u=v_0$, then $M_H(y_i)(v_0)=m_{v_0v_i}-m_{v_0v_0}=-\left(d_{M_H}(W_E)-r_{M_H}(W_E)\right)$. Similarly for $u=v_i$,  $M_H(y_i)(v_i)=m_{v_iv_i}-m_{v_iv_0}=\left(d_{M_H}(W_E)-r_{M_H}(W_E)\right).$
   Therefore, $M_Hy_i=\left(d_{M_H}(W_E)-r_{M_H}(W_E)\right)y_i,$ for all $i=1,2,\ldots,n$. Since $\{y_i:i=1,2,\ldots,n\}$ are linearly independent, the multiplicity of the eigenvalue is at least $|W_{E}|-1$. 
		\end{proof}
{ The above result can also be derived from \cite[Theorem 2.4]{unit}. }
  \begin{rem}
      There may exists two distinct units $W_E,W_F\in\mathfrak{U}(H)$ such that  $d_{M_H}(W_E)-r_{M_H}(W_E)=d_{M_H}(W_F)-r_{M_H}(W_F)=\lambda\text{(~say)}$. We will show such instances shortly in the \Cref{ex-unit-comp}. In that case, as the proof of the \Cref{lem-unit} suggests (see the eigenvectors of the form $y_i=\chi_{\{v_i\}}-\chi_{\{v_0\}}$ for $v_i, v_0 \in W_E$ in the proof), if $y_{W_E} $ is an eigenvector of $\lambda$, then  $y_{W_E}(v)=0$ for all $v\in V(H)\setminus W_E$. Similarly, if $y_{W_F} $ is an eigenvector of $\lambda$, then $y_{W_F}(v)=0$ for all $v\in V(H)\setminus W_F$. Thus, the (usual) inner product $\langle y_{W_E},y_{W_F}\rangle=0$, and $y_{W_E},y_{W_F}$ are linearly independent to each other. Therefore, the multiplicity of $\lambda$ is at least $|W_E|+|W_F|-2$.
  \end{rem}
 We refer to each eigenvalue obtained using \Cref{lem-unit} as \emph{unit eigenvalue}. In the following examples, we show some unit-compatible matrices and provide unit eigenvalues of these matrices using the \Cref{lem-unit}. If $\lambda_{W_E}(M_H)$ is a unit eigenvalue of a unit-compatible matrix $M_H$ corresponding to the unit $ W_{E}=\{v_0,v_1,\ldots,v_n\}$, then the eigenspace of $\lambda_{W_E}$ has a subspace generated by the vectors $\{\chi_{\{v_i\}}-\chi_{\{v_0\}}:i=1,2,\ldots,n\}$. Now onward, we denote this subspace as $S_{W_E}$. It is intriguing to note that this subspace depends only on the units but not on the unit-compatible matrix. That is, for any unit-compatible matrix $M_H$, a unit-eigenvalue $\lambda_{W_E}(M_H)=d_{M_H}(W_E)-r_{M_H}(W_E)$ depends on both the matrix $M_H$, and the unit $M_H$, but the subspace $S_W$ of the eigenspace of $\lambda_{W_E}(M_H)$ depends only on $W_E$ but not on the matrix $M_H$. We have  $y_i=\chi_{\{v_i\}}-\chi_{\{v_0\}}$ is such that $y_i(v)=0$ for all $v\in V(H)\setminus W_E$, and  $\sum\limits_{v\in W_E}y_i(v)=0$ for all $i=1,2,\ldots,n$. Thus the subspace
 $$S_{W_E}=\{x:V(H)\to\mathbb{C}:\sum\limits_{v\in W_E}x(v)=0, \text{~and~} x(v)=0 \text{~for all~}v\in V(H)\setminus W_E\}$$ is spanned by $\{y_i:i=1,\ldots,n\}$, and its dimension is $|W_E|-1$.
 
  \begin{exm}\label{ex-unit-comp}\rm 
In literature, there are several unit-compatible matrices, and as the \Cref{lem-unit} suggests,  the existence of units in the hypergraph is manifested in the spectrum of these matrices.
Recall the adjacency matrix $A^{(r)}_H$ described in the \Cref{ex-hyp-mat}. This matrix is a unit-compatible matrix with $d_{A^{(r)}_H}(W_E)=0$, $s_{W_E}^w(A^{(r)}_H)=|E\cap E_w(H)|$, and $r_{A^{(r)}_H}(W_E)=|E|$ for any $W_E\in \mathfrak{U}(H)$. Thus, for all $W_E\in \mathfrak{U}(H)$, we have an eigenvalue $-|E|$ with multiplicity $|W_E|-1$ of $A^{(r)}_H$. Let $H$ be the hypergraph considered in the \Cref{ex-hyp-rot} (illustrated in the \Cref{fig:rot}). Three units of $H$ are $W_{E_1}=\{2,3\}$, $W_{E_2}=\{5,6\}$, and $W_{E_3}=\{8,9\}$, where the generating sets are $E_1=\{f,i,j\}$, $E_2=\{g,j,h\}$, and $E_3=\{e,h,i\} $. Since $|E_1|=|E_2|=|E_3|=3$, $-3$ is an eigenvalue of multiplicity at least $1$ for each unit $W_{E_i}$, for $i=1,2,3$. The corresponding eigenvectors are $\chi_{\{3\}}-\chi_{\{2\}}$, $\chi_{\{5\}}-\chi_{\{6\}}$, and $\chi_{\{8\}}-\chi_{\{9\}}$.
    
      If $H$ is the hypergraph detailed in the \Cref{hyp-fig} (see the \Cref{fig:unit}), then corresponding to the units $W_{E_1}=\{1,2\}$, $W_{E_2}=\{3,4\}$, $W_{E_3}=\{5,6,15\}$, $W_{E_4}=\{7,8\}$, $W_{E_5}=\{9,10\}$, $W_{E_6}=\{11,12,16\}$, $W_{E_7}=\{13,14\}$, and $W_{E_8}=\{17,18\}$ the unit-eigenvalue of $A^{(r)}_H$ are $-5$, $-3$, $ -1$, $ -1$, $ -1$, $-2$, $ -2$,  and $ -2$ respectively, with multiplicity at least $1$, $1$, $ 2$, $ 1$, $1$, $2$, $1$, and $1$ respectively. Similar fact can be proved for matrices $A^{(b)}_H$, $L_H^{(b)}$, $Q$ described in the \Cref{ex-hyp-mat}.
  \end{exm}
  Let $H$ be a hypergraph and $W_E\in \mathfrak{U}
(H)$. For any hyperedge $e\in E(H)$,
 either $W_E\cap e=\emptyset$ or $W_E\subseteq e$. Consequently, every hyperedge $e\in E(H)$ is either a unit itself or can be represented as a disjoint union of units. By identifying each unit as a single vertex, every hyperedge is contracted, leading to the following notion.
 
		\begin{df}[Unit-contraction]
			Let $H$ be a hypergraph. The unit-contraction of $H$, denoted by $\hat H$, is a hypergraph such that $V(\hat{H})=\mathfrak{U}(H)$, and $E(\hat H)=\{\{W_{E_1},W_{E_2},\ldots,W_{E_{n_e}}\}:e=\bigcup\limits_{i=1}^{n_e}W_{E_i}\in E(H)\}$. Here $n_e$ is the total number of units contained in $e$.
		\end{df}
		Thus, unit contraction leads us to the bijection $\pi_H:E(H)\to E(\hat{H})$ defined by $e=\bigcup\limits_{i=1}^{n_e}W_{E_i}\mapsto \hat{e}=\{W_{E_1},W_{E_2},\ldots,W_{E_{n_e}}\}$ for all $e\in E(H)$. We refer to $\pi_H$ as the hyperedge-unit-contraction.
		
	Let $H$ be a hypergraph, $W_E\in\mathfrak{U}(H)$, and $M_H=\left(m_{uv}\right)_{u,v\in V(H)}$ be a unit-compatible matrix. For  $u,v\in W_E$,
		\begin{align*}
			\sum\limits_{w\in W_E}m_{uw}&=d_{M_H}(W_E)+m_{uv}+\sum\limits_{w\in W_E\setminus\{u,v\}}m_{uw}\\
			&=d_{M_H}(W_E)+m_{vu}+\sum\limits_{w\in W_E\setminus\{u,v\}}m_{vw}= \sum\limits_{w\in W_E}m_{vw}.
		\end{align*}
		Thus, for all $u\in W_E$, the sum $ \sum\limits_{w\in W_E}m_{uw}=b_H(W_E)$ is a constant. For two distinct $W_E,W_F\in \mathfrak{U}(H)$, if $u,v\in W_E$ then
		$$ \sum\limits_{w\in W_F}m_{uw}= \sum\limits_{w\in W_F}m_{vw}.$$
		Thus, for all $u\in W_E$, the sum $ \sum\limits_{w\in W_F}m_{uw}=b_H(W_E,W_F)$, a real number that depends only on $W_E$, and $W_F$.	
  For any unit-compatible matrix $M_H$ associated with a hypergraph $H$,
		the quotient matrix $[M_H/\mathfrak{U}(H)]=\left(b_{ij}\right)_{W_{E_i},W_{E_j}\in\mathfrak{U}(H)}$ is a matrix of order $|\mathfrak{U}(H)|$, such that 
		\begin{align}\label{unit-q}
		    b_{ij}=
		\begin{cases}
			b_H(W_{E_i})&\text{~if~}i=j,\\
			b_H(W_{E_i},W_{E_j})&\text{~otherwise.}
		\end{cases}
		\end{align}

Using \Cref{lem-unit}, we can find $\sum\limits_{W_E\in\mathfrak{U}(H)}|W_E|-1=|V(H)|-|\mathfrak{U}(H)|$ out of $|V(H)|$ eigenvalues of a unit-compatible matrix $M_H$ associated with the hypergraph $H$.  We use 
		the quotient matrix $[M_H/\mathfrak{U}(H)]$ for finding the remaining $|\mathfrak{U}(H)|$ eigenvalues.
	
  For any hypergraph $H$, if $y:V(\hat{H})\to\mathbb{C}$, the blow up of $y$, is a funtion  $\overline{y}:V(H)\to\mathbb{C}$, defined by $\overline{y}(v)=y(W_{E_v(H)})$. 
		\begin{lem}\label{cont-equit-lem}
			Let $H$ be a hypergraph and $M_H$ be a unit-compatible matrix. Each eigenvalue of $[M_H/\mathfrak{U}(H)]$ is an eigenvalue of $M_H$. Moreover, if $[M_H/\mathfrak{U}(H)]y=\lambda y$, then $M_H \overline{y}=\lambda \overline{y}$.
		\end{lem}
		\begin{proof}
			For any $v\in W_{E_i}\in \mathfrak{U}(H)$, 
			\begin{align*}
				(M_H\Bar{y})(v)&=\sum\limits_{w\in V(H)}m_{vw}\Bar{y}(w)=\sum\limits_{W_{E_j}\in \mathfrak{U}(H) }\sum\limits_{w\in W_{E_j} }m_{vw}\Bar{y}(w)\\
				&=\sum\limits_{W_{E_j}\in \mathfrak{U}(H) }b_{ij}y(W_{E_j})=([M_H/\mathfrak{U}(H)]y)(W_{E_i})
			\end{align*}
			Thus, if $[M_H/\mathfrak{U}(H)]y=\lambda y$ then $(M_H(\overline{y}))(v)=([M_H/\mathfrak{U}(H)]y)(W_{E_i})=\lambda y(W_{E_i})=\lambda (\overline{y})(v)$. Therefore, $M_H\overline{y}=\lambda \overline{y}$.
		\end{proof}

  	\subsection{Unit-automorphisms in a hypergraph}
  Given any hypergraph $H$, any symmetry in the unit contraction $\hat{H}$ induces a symmetry in $H$. If $f\in Aut(\hat{H})$, then we refer to the symmetry in $H$ associated with $f$ as unit-automorphism in $H$.
		\begin{df}[Unit-automorphism]
			A \emph{unit-automorphism} is a  bijection $$f:\mathfrak{U}(H)\to\mathfrak{U}(H)$$ such that $$e=\bigcup\limits_{i=1}^{n_e}W_{E_i}\in E(H)$$ if and only if $$e^{\prime} =\bigcup\limits_{i=1}^{n_e}f(W_{E_i})\in E(H).$$
			
		\end{df}
		
		That is, a unit-automorphism is a bijection $f$ on $\mathfrak{U}(H)$ that induces a bijection $\hat{f}$ on $E(H)$ by
		$$e=\bigcup\limits_{i=1}^{n_e}W_{E_i}(\in E(H))\mapsto e^{\prime} =\bigcup\limits_{i=1}^{n_e}f(W_{E_i})(\in E(H)).$$

		\begin{prop}\label{hyp-aut}
			Every hypergraph automorphism induces a unit-automorphism.
		\end{prop}
		\begin{proof}
			Let $H$ be a hypergraph and $g:V(H)\to V(H)$ is a hypergraph automorphism.
			
			Since $g$ is a bijection, for any $W_{E_0}\in\mathfrak{U}(H)$, for any $u,v\in W_{E_0}$,
			$$E_{g(u)}(H)=E_{g(v)}(H).$$
			Therefore, using $g$, we can define the bijection 
			
			$$f:\mathfrak{U}(H)\to\mathfrak{U}(H)$$
			
			as $$W_{E_0}\mapsto W_{E_{g(v)}(H)},$$
			where $v\in W_{E_0}$.
			
			Since $g$ is a hypergraph automorphism, $$e=\bigcup\limits_{i=1}^{n_e}W_{E_i}\in E(H)$$ if and only if $$e^{\prime} =\bigcup\limits_{i=1}^{n_e}f(W_{E_i})=\{g(v):v\in e\}\in E(H).$$
			
		\end{proof}
		The converse of the \Cref{hyp-aut} is not true; for justification, consider the following example.
		\begin{exm}\label{uit-auto-non-auto}\rm
			Consider the hypergraph $H$ given in \Cref{hyp-fig} (see \Cref{fig:unit}). If we define $f:\mathfrak{U}(H)\to\mathfrak{U}(H)$ as 
			$$f=\begin{pmatrix}
				W_{E_1}&W_{E_2}&W_{E_3}&W_{E_4}&W_{E_5}&W_{E_6}&W_{E_7}&W_{E_8}
				\\W_{E_1}&W_{E_2}&W_{E_4}&W_{E_5}&W_{E_3}&W_{E_7}&W_{E_6}&W_{E_8}
			\end{pmatrix}.$$
			It can be easily verified that $f$ induced a bijection $\hat{f}:E(H)\to E(H)$, given by
			$$ \hat{f}=
			\begin{pmatrix}
				e_1&e_2&e_3&e_4&e_5&e_6&e_7\\
				e_3&e_1&e_2&e_5&e_4&e_7&e_6
			\end{pmatrix}.$$
			Thus, $f$ is a unit-automorphism. Since $f$ can only induce a bijection on $E(H)$ but not on $V(H)$, the unit-automorphism $f$ can not induce a hypergraph automorphism.
		\end{exm}
		
		If a unit-automorphism $f$ of a hypergraph $H$ is cardinality preserving that is  $|f(W_E)|=|W_E|$ for all $W_E\in \mathfrak{U}(H)$. Thus, we have a bijection $f_{W_E}:W_E\to F(W_E) $ for each unit. Since $V(H)=\bigcup\limits_{W_E\in\mathfrak{U}(H)}W_E$ the collection of bijections  $\{f_{W_E}:W_E\in\mathfrak{U}(H)\}$ leads us to the hypergraph automorphism $\bar{f}:V(H)\to V(H) $ defined by $f(v)=f_{W_E}$ if $v\in W_E$ for any $W_E\in\mathfrak{U}(H)$. Thus, if a unit-automorphism is cardinality preserving, then it induces a hypergraph automorphism, but if a unit-automorphism does not preserve the cardinality of units, then it may not induce any automorphism.	For instance, in \Cref{uit-auto-non-auto}, the cardinality of $W_{E_3}$ and $W_{E_5}$ are not the same, and the unit automorphism $f$ does not induce any automorphism.
  
		Unit automorphisms are such symmetries of hypergraph that are preserved under unit contraction. Thus, we have the following result.
		
	\begin{lem}
			A unit-automorphism in $H$ corresponds to a hypergraph automorphism in $\hat H$. Any unit rotation of order $p$ on a hypergraph $H$ is a rotation in $\hat H$ of order $p$.
		\end{lem}
		\begin{proof}
			Let $f:\mathfrak{U}(H)\to\mathfrak{U}(H)$ be a unit-automorphism. Thus, it induces a bijection $\hat{f}:E(H)\to E(H)$. The bijection $\hat{f}$ and the hyperedge-unit- contraction 
			$\pi_H $ lead us to another bijection $ \pi_H\circ\hat{f}\circ\pi_H^{-1}:E(\hat{H})\to E(\hat{H})$.
			\begin{center}

				\begin{tikzpicture}[scale=0.5]
					
					\node [style=none] (0) at (0, 4) {$E(H)$};
					\node [style=none] (1) at (0, 0) {$E(\hat{H})$};
					\node [style=none] (2) at (0, 3.5) {};
					\node [style=none] (3) at (0, 0.5) {};
					\node [style=none] (4) at (7.5, 4) {$E(H)$};
					\node [style=none] (5) at (7.5, 0) {$E(\hat H)$};
					\node [style=none] (6) at (7, 3.5) {};
					\node [style=none] (7) at (7, 0.5) {};
					\node [style=none] (8) at (3.75, 4.75) {$\Hat f$};
					\node [style=none] (9) at (-0.75, 2) {$\pi_H$};
					\node [style=none] (10) at (8, 2.25) {$\pi_H$};
					\node [style=none] (11) at (4, -0.75) {$\pi_H\circ\hat{f}\circ\pi_H^{-1}$};
					\node [style=none] (12) at (0.75, 4) {};
					\node [style=none] (13) at (6.25, 4) {};
					\node [style=none] (14) at (0.75, 0) {};
					\node [style=none] (15) at (6.25, 0) {};
					\draw [->] (2.center) to (3.center);
					\draw [->] (6.center) to (7.center);
					\draw [->] (12.center) to (13.center);
					\draw [->] (14.center) to (15.center);
				\end{tikzpicture}
			\end{center}
			Since $V(\hat{H})=\mathfrak{U}(H)$, the bijection $f:\mathfrak{U}(H)\to\mathfrak{U}(H)$ on $V(\hat{H})$ induces a bijection $\pi_H\circ\hat{f}\circ\pi_H^{-1}$ on $E(\hat{H})$. Therefore, $f$ is a hypergraph automorphism of the hypergraph $\hat{H}$. Similarly, we can prove any unit rotation of order $p$ on a hypergraph $H$ is a rotation in $\hat H$ of order $p$.
		\end{proof}
		Now, we can extend the notion of automorphism-compatible matrices to \emph{unit-automorphism-compatible} matrices.
		\begin{df}[Unit-automorphism-compatible matrices ]
			Let $H$ be a hypergraph. Suppose that $f:V(\hat H)\to V(\hat H)$ is unit automorphism in $V(H)$. A unit-compatible matrix $M_H$, indexed by $V(H)$, is called $f$-compatible if the matrix $[M_H/\mathfrak{U}(H)]$ is $f$-compatible with respect to the automorphism $f:V(\hat H)\to V(\hat H)$ associated with $\hat H$. 
		\end{df}

Given a unit-automorphism \( f \), a unit-compatible matrix might not necessarily be \( f \)-compatible. A clear example of this is observed in the unit-automorphism \( f \) detailed in \Cref{uit-auto-non-auto}, where the matrix \( A_H^{(r)} \) fails to be \( f \)-compatible. This is evident because when \( f(W_{E_3}) = W_{E_4} \) and \( f(W_{E_1}) = W_{E_1} \), the entry corresponding to the edge \( (W_{E_1},W_{E_3}) \) in \( [A_H^{(r)}/\mathfrak{U}(H)] \) is \( 3 \), while the entry corresponding to \( (W_{E_1},W_{E_4}) \) is \( 2 \). Consequently, the matrix \( [A_H^{(r)}/\mathfrak{U}(H)] \) is not compatible with the automorphism \( f:V(\hat{H})\to V(\hat{H}) \) of the hypergraph \( \hat H \). Thus, \( A_H^{(r)} \) is not compatible with the unit automorphism \( f:\mathfrak{U}(H)\to \mathfrak{U}(H) \). The next result provides some unit automorphisms such that unit-compatible matrices are compatible with these unit automorphisms.
	A  \emph{unit saturated} set $U\subseteq V(H)$ is such a set that for any $W_E\in \mathfrak{U}(H)$, if $U\cap W_{E}\ne \emptyset$, then $W_E\subseteq U$. 
 For instance, any $e\in E(H)$ can be written as union of units: $e=\bigcup\limits_{v\in e}W_{E_v(H)}$.
\begin{prop}\label{prop-unit-comp}
    Let $H$ be a hypergraph. If a unit-automorphism $f:\mathfrak{U}(H)\to \mathfrak{U}(H)$ of $H$ is cardinality preserving, that is $|f(W_E)|=|W_E|$ for all $W_E\in \mathfrak{U}(H)$, then $f$ induces a hypergraph automorphism $\bar f\in Aut(H)$. If a unit-compatible matrix $M_H$ is also compatible with the hypergraph automorphism $\bar f$, then $M_H$ is also compatible with the unit-automorphism $f$.
\end{prop}
\begin{proof}
    Since $|f(W_E)|=|W_E|$ for all $W_E\in \mathfrak{U}(H)$, there exist bijection $f_{W_E}:W_E\to f(W_E)$. Since the vertex can be partitioned into disjoint units as $V(H)=\bigcup\limits_{W_E\in\mathfrak{U}(H)}W_E$, for all $v\in V(H)$, we can define $\Bar{f}:V(H)\to V(H)$ as $\Bar{f}(v)=f_{W_E}(v)$ if $v\in W_E$ for some $W_E\in\mathfrak{U}(H)$. Since each $f_{W_E}$ is a bijection, the map $\bar{f}$ is a bijection. For all $e\in E(H)$, the hyperedge $e$ can be expressed as disjoint union of units as $e=W_{E_1}\cup W_{E_2}\cup\ldots\cup W_{E_{n_e}}$, for some natural number $n_e$. Since $f$ is a unit-automorphism of $H$, we have $\hat{\bar f}(e)=\{f(v):v\in e\}=f(W_{E_1})\cup f(W_{E_2})\cup\ldots\cup f(W_{E_{n_e}})=e'\in E(H)$  if and only if $e=W_{E_1}\cup W_{E_2}\cup\ldots\cup W_{E_{n_e}}\in E(H)$. This leads to the bijection of hyperedges $\hat{\bar f}:E(H)\to E(H) $ defined by $ \hat{\bar f}(e)=\{f(v):v\in e\}$. Therefore $\hat{\bar f}\in Aut(H)$. 

    Suppose that the unit-compatible matrix $M_H=\left(m_{uv}\right)_{u,v\in V(H)}$ is compatible with $\bar{f}\in Aut(H)$, and the quotient matrix is $[M_H/\mathfrak{U}(H)]=\left(b_{ij}\right)_{W_{E_i},W_{E_j}\in\mathfrak{U}(H)}$, as defined in \Cref{unit-q}. For $W_E,W_F\in\mathfrak{U}(H)$, and for some $ u\in W_E$, and $\bar f(u)\in f(W_E)$,
    \begin{align*}
        b_H(W_E,W_F)=\sum\limits_{v\in W_F}m_{uv}=\sum\limits_{\bar f(v)\in f(W_F)}m_{\bar f(u)\bar f(v)}=b_H(f(W_E),f(W_F)).
    \end{align*}
    Similarly, $b_H(W_E)=b_H(f(W_E))$ for all $W_E\in \mathfrak{U}(H)$. Thus, $M_H$ is compatible with the unit-automorphism $f$.
\end{proof}
In the \Cref{ex-hyp-mat}, we have described some matrices that are compatible with any hypergraph automorphism. The \Cref{ex-unit-comp} shows that some of them are unit-compatible. Consequently, by the \Cref{prop-unit-comp}, these matrices are compatible with any cardinality preserving unit-automorphism $f$.

In the next example, we show a matrix $M_H$ such that $M_H$ is compatible with unit-automorphisms even if it is not cardinality preserving.
\begin{exm}\rm
    Let $H$ be a hypergraph and $M_H=\left(m_{uv}\right)_{u,v\in V(H)}$ be a matrix such that for two distinct $u,v\in V(H)$,
    $m_{uv}=\frac{|E_u(H)\cap E_v(H)|}{n_v}$, and for $u=v$, the diagonal entry is $m_{uu}=\frac{|E_u(H)|}{n_u}$, where $n_u=|\{v:E_v(H)=E_u(H)\}|$ for any vertex $u\in V(H)$. Consider the quotient matrix $[M_H/\mathfrak{U}(H)]=\left(b_{W_{E_i}W_{E_j}}\right)_{W_{E_i},W_{E_j}\in\mathfrak{U}(H)}$ defined in \Cref{unit-q}. Therefore, $b_{W_{E_i}W_{E_j}}=|E_i\cap E_j|$ for $i\ne j$, and the diagonal entries are $b_{W_{E_i}W_{E_i}}=|E_i|$ for all $i$. Being an unit automorphism of $H$, $f\in Aut(\hat{H})$, for all $W_E\in\mathfrak{U}(H)$, if $W_{E'}=f(W_E)$, and $W_{F'}=f(W_F)$ then $|E'|=|E|$, and $|E'\cap F'|=|E\cap F|$. Consequently, $b_{W_{E}W_{F}}=b_{f(W_{E})f(W_{F})}$ for all $W_{E},W_F\in \mathfrak{U}(H)$, and $[M_H/\mathfrak{U}(H)]=\left(b_{W_{E_i}W_{E_j}}\right)_{W_{E_i},W_{E_j}\in\mathfrak{U}(H)}$ is compatible with the automorphism $f$ of $\hat{H}$, and thus $M_H$ is compatible with the unit automorphism $f$ of $H$.
\end{exm}
Given any unit-automorphism $f$ of a hypergraph $H$, by the \Cref{aut-decomp-rot} there exists disjoint rotations $f_1,f_2,\ldots, f_k$ in the unit-contraction $\hat{H}$ such that $f=f_1\circ f_2\circ\ldots\circ f_k$. We refer to this decomposition as the unit-rotation decomposition of the unit automorphism $f$. In the next result, we provide the eigenvalues of unit-automorphism-compatible matrices in terms of smaller matrices associated with the unit-rotation decomposition of the unit automorphism. 
\begin{thm}
    Let $H$ be a hypergraph and $f$ is a unit-automorphism in $H$. Suppose that $f=f_1\circ f_2\circ\ldots \circ f_k$ is a unit-rotation decomposition of $f$. If $M_H=\left(m_{uv}\right)_{u,v\in V(H)}$ is $f$-compatible, then the eigenvalues of $M_H$ are the following.
    \begin{enumerate}
				\item All the unit eigenvalues of $M_H$, that is $\{d_{M_H}(W_E)-r_{M_H}(W_E):W_E\in\mathfrak{U}(H)\}$,
				\item  eigenvalues of the matrix $N_{\hat H}^{\omega_{ij}}(f_i,U_0^{(i)})$ for all $j=1,\ldots,l_i-1$, $i=1,\ldots,k$, and the quotient matrix $[N_{\hat H}/\mathcal{O}_f]$,
			\end{enumerate}
   where $N_{\hat H}=[M_H/\mathfrak{U}(H)]$ and $U_0^{(i)} $ is the $0$-th component of the rotation $f_i$ of $\hat H$.
\end{thm}\begin{proof}
    Since $M_H$ is compatible with the unit-automorphism $f$, the matrix $M_H$ is unit-compatible. Consequently, by the \Cref{lem-unit} each element of the collection $\{d_{M_H}(W_E)-r_{M_H}(W_E):W_E\in\mathfrak{U}(H)\}$ is an eigenvalue of $M_H$.

    The total number of unit eigenvalues of $M_H$ are $|V(H)|-|\mathfrak{U}(H)|$. By the \Cref{cont-equit-lem} the $|\mathfrak{U}(H)|$ number of eigenvalues of $M_H$ are the eigenvalues of $[M_H/\mathfrak{U}(H)] $. The proof of the \Cref{cont-equit-lem} shows that if $[M_H/\mathfrak{U}(H)]y=\mu y$ then $M_H\bar y=\mu \bar y$. The eigenvector $\bar y$ is constant on each $W_E\in \mathfrak{U}(H)$. If $x$ is the eigenvector of some unit-eigenvalue corresponding to a unit $W_E$, then $x\in S_{W_E}$, and thus $\sum\limits_{v\in W_E}x(v)=0$, and $x(v)=0$ for all $v\in V(H)\setminus W_E$. Therefore, $\langle x,y\rangle=0$, and $x,y$ are linearly independent. Therefore, the complete spectrum of $M_H$ consists of all the unit eigenvalues of $M_H$ and all the eigenvalues of $N_{\hat H}=[M_H/\mathfrak{U}(H)]$.
    By the \Cref{aut-comp}, since $N_{\hat H}$ is compatible with the automorphism $f:V(\hat H)\to V(\hat H)$, the eigenvalues of $N_H$ are the eigenvalues of the matrices $N_{\hat H}^{\omega_{ij}}(f_i,U_0^{(i)})$ for all $j=1,\ldots,l_i-1$, $i=1,\ldots,k$, and the quotient matrix $[N_{\hat H}/\mathcal{O}_f]$.
\end{proof}
\section{Discussion and Conclusion}
Symmetries in hypergraphs can lead us to invariant subspaces of hypergraph matrices. We refer the reader \cite{invariant_2024} for a detailed study on invariant subspaces of hypergraph matrices due to hypergraph symmetries.
 Let $H$ be a hypergraph, and $f$ is a rotation in $H$ of order $n$. Suppose that the $0$-th component of $f$ is $U_0=\{v_1,v_2,\ldots,v_m\}$, and $X=\{u_1,u_2,\ldots,u_p\}$ is the invariant set under $f$. In that case, the collection of $f$-orbits are $$\mathcal{O}_f=\{O_f(v_1),O_f(v_2),\ldots,O_f(v_m),O_f(u_1),O_f(u_2),\ldots,O_f(u_p)\}.$$ Given any $v\in V(H)$ the characteristic function  $\chi_{_{O_f(v)}}:V(H)\to\mathbb{R}$ of the orbit $O_f(v)$ of $v$ is such that $\chi_{_{O_f(v)}}(u)=1$ if $u\in O_f(v)$, and otherwise $\chi_{_{O_f(v)}}(u)=0$. Consider the $(m+p)$-dimensional subspace $\mathcal{V}_f=\{x:V(H)\to\mathbb{R}:x=\sum\limits_{i=1}^mc_i\chi_{_{O_f(v_i)}}+\sum\limits_{j=1}^{p}c_{i+m}\chi_{_{O_f(u_i)}}\}$. For all $x\in \mathcal{V}_f$, the function $x$ is constant on each $f$-orbit. Thus, we name $\mathcal{V}_f$ as \emph{$f$-orbit synchronized space}. Given any $f$-compatible matrix $M_H$, if $\{z^{(i)}:i=1,2,\ldots,m+p\}$ is the complete list of eigenvectors of the $[M_H/\mathcal{O}_f]$, then the collection $\mathcal{Z}=\{z^{(i)}_{O_f}:i=1,2,\ldots,m+p\}$ spans the subspace $\mathcal{V}_f$. Consequently, $\mathcal{V}_f$ is an invariant subspace of any $f$-compatible matrix $M_H$. 
 Thus, given any $f$-compatible matrix $M_H$, the $f$-orbit synchronized space, $\mathcal{V}_f$ is an invariant subspace of $M_H$.
 Suppose that  $f$ is a rotation of order $n$ in a hypergraph $H$, and $M_H$ is any $f$ compatible matrix. If $\{x_n: V(H)\to \mathbb{C}\}_{n\in\mathbb{N}}$ is a sequence of functions governed by the following iterative equation
   \begin{equation}\label{iteration}
       x_{n+1}=M_Hx_n.
   \end{equation}
   If for some $m\in\mathbb{N}$, there is a cluster synchronization in all the $f$-orbits, this cluster synchronization is retained in the subsequent members of the sequence. That is, for any orbit $O_f(u)\in \mathcal{O}_f$, if $x_m(v)=c_{u,m}$, a constant for all $v\in O_f(u)$, then for any orbit $O_f(u)\in \mathcal{O}_f$, $x_n(v)=c_{u,n}$ for all $v\in O_f(u)$. The iterative equations like \Cref{iteration} are common in dynamical networks and random walks. Therefore, the underlying hypergraph has a rotation $f$, then $f$ induces cluster-synchronization in $f$-orbits.

	



\end{document}